\documentclass[10pt]{article}
\usepackage{amssymb,amsfonts}

\usepackage[dvips]{graphicx}
\usepackage{psfrag}
\usepackage{ifthen}
\usepackage[usenames]{color}
\usepackage{caption}
\usepackage{pinlabel}
\usepackage{relsize}

\usepackage{amssymb}
\usepackage{latexsym}
\usepackage{amsmath}
\usepackage{mathrsfs}
\usepackage[all]{xy}
\usepackage{enumerate}
\usepackage{amsthm}
\usepackage{amsmath,amsthm,amsfonts,amssymb}
\usepackage[usenames]{color}

\newtheorem{theorem}{Theorem}
\newtheorem{definition}{Definition}
\newtheorem{example}{Example}
\newtheorem{lemma}{Lemma}
\newtheorem{cor}{Corollary}
\newtheorem{remark}{Remark}

\newtheorem{prop}{Proposition}

\newcommand{\be}{\begin{enumerate}}
\newcommand{\ee}{\end{enumerate}}
\newcommand{\beq}{\begin{equation}}
\newcommand{\eeq}{\end{equation}}
\usepackage{amssymb}

\usepackage{amsmath}
\usepackage{amssymb}

\newcommand{\gst}{\, | \,}  
\newcommand{\naturals}{\ensuremath{\mathbb{N}}} 
\newcommand{\GammaPresentation}{\langle A\,|\,\mathcal{R}\rangle} 
\newcommand{\GPresentation}{\ensuremath{\langle Z\gst S\rangle}}

\newcommand{\mo}{{-1}}
\newcommand{\bel}[1]{\begin{equation}\label{#1}}
\newcommand{\eee}{\end{equation}}

\newcommand{\LBA}{\left\{\begin{array}}
\newcommand{\EAR}{\end{array}\right.}
\input epsf
\newcommand{\prm}[1]{#1^{\prime}}
\newcommand{\Complement}[1]{#1^{\mathsf{c}}}

\newenvironment{romanenumerate}
	{\begin{enumerate}

	}
        {
        
        \end{enumerate}
        } 

\newcommand{\braced}[4]{\left\{\begin{array}{ll} {#1} & {#2} \\ {#3} & {#4} \end{array}\right.}
\newcommand{\bracedTwo}{\braced}

\newcommand{\comment}[1]{} 

\begin{document}
\title{Quadratic Equations in Hyperbolic Groups are NP-complete}
\author{Olga Kharlampovich, Atefeh Mohajeri, Alexander Taam, Alina Vdovina}

\maketitle
\begin{abstract}
We prove that in a torsion-free hyperbolic  group $\Gamma$,
the length of the value of each variable in a minimal  solution of  a  quadratic equation $Q=1$ is bounded by $N|Q|^3$ for an orientable equation, and by $N|Q|^{4}$  for a non-orientable equation, where $|Q|$ is the length of the equation, and the constant $N$ can be computed.  We show that  the problem, whether a quadratic equation in $\Gamma$ has a solution, is in NP, and that there is a PSpace algorithm for solving arbitrary equations in $\Gamma$. If additionally $\Gamma$ is non-cyclic, then this problem (of deciding existence of a solution) is NP-complete. We also give a slightly larger bound for minimal solutions of quadratic equations in a toral relatively hyperbolic group.  \end{abstract}

\section{Introduction}
The study of quadratic equations (an equation is quadratic if each variable appears exactly twice) over free groups began  with the
work of Malcev \cite{Malcev-1962}.  One of the reasons the research in this topic has been so
fruitful is a deep connection between quadratic equations and the
topology of surfaces (see, for example \cite{Grigorchuk-Kurchanov-1989}).

In \cite{Comerford-Edmunds-1981} the problem of deciding if a quadratic
equation over a free group is satisfiable
was shown to be decidable. In
addition it was shown in \cite{Olshanskii-1989},
\cite{Grigorchuk-Kurchanov-1989}, and \cite{Grigorchuk-Lysenok-1992}
that if $n$, the number of variables, is fixed, then deciding if a standard
quadratic equation has a solution can be done in time which is polynomial in
the sum of the lengths of the coefficients.  (In \cite{Grigorchuk-Lysenok-1992}
 this was shown for hyperbolic groups.) In
\cite{KLMT} it was shown that  the problem whether a quadratic equation in a free group has a solution, is NP-complete. In the present paper we will show that similar problem is NP-complete in a torsion-free hyperbolic group and obtain polynomial bounds on minimal solutions of quadratic equations in such a  group. There are still very few problems in topology and geometry which are known to be NP-complete,  one of them is  3-manifold knot genus \cite{AHT}. 
It was proved in \cite{kv} that in a free group,
the length of the value of each variable in a minimal solution of a standard  quadratic equation is bounded by $2s$ for an orientable equation and by $12s^4$ for a non-orientable equation, where $s$ is the sum of the lengths of the coefficients.
Similar result was proved in \cite{lm} for arbitrary quadratic equations.

If a group $G$ has a generating set $A$, we will consider a solution of a system of equations $S(X,A)=1$ in $G$ as a $G$-homomorphism $\phi : (F(X)*G)/ncl (S)\rightarrow G$.  The length $|S|$ of the system $S(X,A)=1$ is the sum of the lengths of equations in $S(X,A)=1$ considered as elements of  the free group $F(A,X)$.
We will prove the following results.

\begin{theorem} \label{th:1} Let $\Gamma$ be a torsion-free hyperbolic group given by a generating set $A$ and a finite number of relations. It is possible to compute a constant $N$ with the property that if a quadratic equation $Q (X,A) = 1$ is solvable in  $\Gamma$,
 then there exists a solution $\phi$ such that for any variable $ x$,
$|\phi (x)|\leq N|Q|^3$ if $Q$ is orientable,
$|\phi (x)|\leq N|Q|^{4},$ if $Q$ is non-orientable.

\end{theorem}

A group $G$ that is hyperbolic relative to a collection $\{H_{1},\ldots,H_{k}\}$ of subgroups (see Section 4 for a definition) is called \emph{toral} if $H_{1},\ldots,H_{k}$ are all abelian and $G$ is
torsion-free.

\begin{theorem} \label{th:2}  Let $\Gamma$ be a  toral relatively  hyperbolic group with generating set $A$. It is possible to compute a constant $N$ with the property that if a quadratic equation $Q (X,A) = 1$ is solvable in  $\Gamma$,
 then there exists a solution $\phi$ such that for any variable $ x$,
$|\phi (x)|\leq N|Q|^5 $ if $Q$ is orientable,
$|\phi (x)|\leq N|Q|^{12},$ if $Q$ is non-orientable.

\end{theorem}

\begin{theorem} \label{NPhard}  Let $\Gamma$ be a non-cyclic torsion-free hyperbolic group. The problem, whether a quadratic equation has a solution in $\Gamma$, is NP-complete.\end{theorem}
Recall that all non-trivial virtually cyclic torsion-free hyperbolic groups must actually be infinite cyclic. So ``non-cyclic'' and ``non-elementary'' may be used interchangeably in this context.

\section{Reduction of equations in hyperbolic groups to equations in free groups}
In \cite{RS95}, the problem of deciding whether or not a system of equations $S(Z)=1$ (we will often just write system $S(Z)$ skipping the equality sign) over a torsion-free hyperbolic group $\Gamma$ has a solution was solved by constructing
\emph{canonical representatives} for certain elements of $\Gamma$. This construction reduced the problem to deciding the existence of solutions in finitely many
systems of equations over free groups, which had been previously solved in \cite{Mak}.  The reduction is described
below.

Let $\overline{\phantom{c}}$ denote the canonical epimorphism $F(Z,A)\rightarrow \Gamma_{S}$, where  $\Gamma_{S}$ is the  quotient of $F(Z)\ast\Gamma$ over the
normal closure of the set $S$.

For a homomorphism $\phi: F(Z,A)\rightarrow K$ we define $\overline{\phi}: \Gamma_{S} \rightarrow K$
by
\[
\big(\overline{w}\big)^{\overline{\phi}} =  w^{\phi},
\]
where any preimage $w$ of $\overline{w}$ may be used.  We will always ensure that $\overline{\phi}$ is a well-defined homomorphism (this is equivalent to fact that $\phi$ factors through $\Gamma _S$).
For a system $S(Z)=1$ without coefficients, $\overline{\phantom{c}}$ denotes the canonical epimorphism $F(Z)\rightarrow \langle Z\gst S\rangle$
and $\overline{\phi}$ is defined analogously.

\begin{prop}\label{Lem:RipsSela1}
Let $\Gamma=\GammaPresentation$ be a torsion-free $\delta$-hyperbolic group and $\pi : F(A)\rightarrow \Gamma$ the canonical epimorphism.  There
is an algorithm that, given a system $S(Z,A)=1$
of equations over $\Gamma$, produces finitely many systems of
equations
\begin{equation}
S_{1} (X_{1},A)=1,\ldots,S_{n}(X_{n},A)=1
\end{equation}
over $F=F(A)$ and homomorphisms $\rho_{i}: F(Z,A)\rightarrow F_{S_{i}}$ for $i=1,\ldots,n$
such that
\begin{romanenumerate}
\item [1)]for every $F$-homomorphism $\phi : F_{S_{i}}\rightarrow F$,  the map $\overline{\rho_{i}\phi\pi}:\Gamma_{S}\rightarrow \Gamma$ is a $\Gamma$-homomorphism, and
\item [2)]for every $\Gamma$-homomorphism $\psi: \Gamma_{S}\rightarrow \Gamma$ there is an integer $i$ and an $F$-homomorphism
$\phi : F_{S_{i}}\rightarrow F(A)$ such that $\overline{\rho_{i}\phi\pi}=\psi$.
\end{romanenumerate}
Further, if $S(Z)=1$ is a system without coefficients, the above holds with $G=\GPresentation$ in place of $\Gamma_{S}$ and `homomorphism' in place of
`$\Gamma$-homomorphism'.

Moreover, $|S_i|=O(|S|^2)$ for each $i=1,\ldots, n$, where the constants depend on the group $\Gamma$.
\end{prop}

\begin{proof}
The result is an easy corollary of Theorem~4.5 of \cite{RS95}, but we will provide a few details.

Assume that the system $S(Z,A)$, in variables $z_{1},\ldots,z_{l}$, consists of $r$ constant equations and $q-r$ triangular equations, i.e.
\[
S(Z,A) = \braced{z_{\sigma(j,1)}z_{\sigma(j,2)}z_{\sigma(j,3)}=1}{j=1,\ldots,q-r}{z_{s}  =  a_{s}}{s=l-r+1,\ldots,l}
\]
where $\sigma(j,k)\in\{1,\ldots,l\}$ and $a_{i}\in\Gamma$.  An algorithm is described in \cite{RS95}
which, for every $m\in\naturals$, assigns to each element $g\in \Gamma$ a word $\theta_{m}(g)\in F$ satisfying
\[
\theta_{m}(g)=g \mbox{ in } \Gamma
\]
called its \emph{canonical representative}.  The representatives $\theta_{m}(g)$   satisfy
useful properties for certain $m$ and certain finite subsets of $\Gamma$ \cite{RS95}, as follows.

Let\footnote{The constant of hyperbolicity $\delta$ may be computed from a presentation of $\Gamma$ using the results of \cite{EH01}.} $L=q\cdot 2^{5050(\delta+1)^{6}(2|A|)^{2\delta}}$.
Suppose $\psi: F(Z,A)\rightarrow \Gamma$ is a solution of $S(Z,A)$ and denote
\[
\psi(z_{\sigma(j,k)})=g_{\sigma(j,k)}.
\]
Then there exist
$h_{k}^{(j)}, c_{k}^{(j)}\in F(A)$ (for $j=1,\ldots,q-r$ and $k=1,2,3$) such that
\begin{romanenumerate}
\item each $c_{k}^{(j)}$ has length less than\footnote{The bound of $L$ here, and below, is from \cite{RS95}.} $L$ (as a word in $F$),  \label{RepsCond1}
\item $c_{1}^{(j)}c_{2}^{(j)}c_{3}^{(j)}  =  1$ in $\Gamma$, \label{RepsCond2}
\item there exists $m\leq L$\ such that the canonical representatives satisfy the following equations in $F$:\label{RepsCond3}
\begin{eqnarray}
\theta_{m} (g_{\sigma(j,1)}) & = & h_{1}^{(j)} c_{1}^{(j)} \left(h_{2}^{(j)}\right)^{-1} \label{CanonReps1}\\
\theta_{m} (g_{\sigma(j,2)}) & = & h_{2}^{(j)} c_{2}^{(j)} \left(h_{3}^{(j)}\right)^{-1}\\
\theta_{m} (g_{\sigma(j,3)}) & = & h_{3}^{(j)} c_{3}^{(j)} \left(h_{1}^{(j)}\right)^{-1}.\label{CanonReps3}
\end{eqnarray}
\end{romanenumerate}
In particular, when $\sigma(j,k)=\sigma(j',k')$ (which corresponds to two occurrences in $S$ of the variable $z_{\sigma(j,k)}$) we have
\begin{equation}
h_{k}^{(j)} c_{k}^{(j)} \left(h_{k+1}^{(j)}\right)^{-1} = h_{k'}^{(j')} c_{k'}^{(j')} \left(h_{k'+1}^{(j')}\right)^{-1}.\label{Hequality}
\end{equation}

Consequently, we construct the systems $S_i(X_{i},A)$ as follows.
For every positive integer $m\leq L$ and every choice of $3(q-r)$ elements $c_{1}^{(j)},c_{2}^{(j)},c_{3}^{(j)}\in F$  ($j=1,\ldots,q-r$)
satisfying (i) and (ii)\footnote{The word problem in hyperbolic groups is decidable.}
we build a system $S_i(X_{i},A)$
consisting of the equations
\begin{eqnarray}
x_{k}^{(j)}c_{k}^{(j)}\left(x_{k+1}^{(j)}\right)^{-1} & = & x_{k'}^{(j')}c_{k'}^{(j')}\left(x_{k'+1}^{(j')}\right)^{-1} \label{Eqn:SC1}\\
x_{k}^{(j)}c_{k}^{(j)}\left(x_{k+1}^{(j)}\right)^{-1} & = & \theta_{m}(a_s) \label{Eqn:SC2}
\end{eqnarray}
where an equation of type   (\ref{Eqn:SC1}) is included whenever  $\sigma(j,k)=\sigma(j',k')$ and an equation of type (\ref{Eqn:SC2}) is included whenever
$\sigma(j,k)=s\in\{l-r+1,\ldots,l\}$.  To define $\rho_{i}$, set
\[
\rho_i (z_s) = \bracedTwo{x_{k}^{(j)}c_{k}^{(j)}\left(x_{k+1}^{(j)}\right)^{-1},}{1\leq s \leq l-r \mbox{ and } s=\sigma(j,k)}{ \theta_{m}(a_s),}{l-r+1\leq s \leq l}
\]
where for $1\leq s \leq l-r$ any $j,k$ with $\sigma(j,k)=s$ may be used.

If $\psi:F(Z)\rightarrow \Gamma$ is any solution to $S(Z,A)=1$, there is a system $S(X_{i},A)$ such that $\theta_{m}(g_{\sigma(j,k)})$ satisfy
(\ref{RepsCond1})-(\ref{RepsCond3}).  Then the required solution $\phi$ is given by
\[
\phi\big(x_{j}^{(k)}\big) = h_{j}^{(k)}.
\]
Indeed, (iii) implies that $\phi$ is a solution to $S(X_{i},A)=1$.  For $s=\sigma(j,k)\in\{1,\ldots,l-r\}$,
\[
z_{s}^{\rho_{i}\phi} = h_{k}^{(j)} c_{k}^{(j)} \left(h_{k+1}^{(j)}\right)^{-1} = \theta_{m}(g_{\sigma(j,k)})
\]
and similarly for $s\in\{l-r+1,\ldots,l\}$, hence $\psi= \rho_{i}\phi\pi$.

Conversely, for any solution $\phi\big(x_{j}^{(k)}\big)= h_{j}^{(k)}$ of $S(X_{i})=1$ one sees that by (\ref{Eqn:SC1}),
\[
z_{\sigma(j,1)}z_{\sigma(j,2)}z_{\sigma(j,3)} \rightarrow^{\rho_{i}\phi} h_{1}^{(j)} c_{1}^{(j)}c_{2}^{(j)}c_{3}^{(j)} \big(h_{1}^{(j)}\big)^{-1}
\]
which maps to 1 under $\pi$ by (ii), hence $\rho_{i}\phi\pi$ induces a homomorphism.

The statement about the length of the systems $S_i=1$ will follow from the next proposition.
\end{proof}

\begin{prop} \label{Bnd}
Let $S=S(Z,A)=1$ be a system of equations over $\Gamma=\GammaPresentation$. Then, for the systems $S_i=S_i(X_i,A)$ defined by equations  (\ref{Eqn:SC1}) and (\ref{Eqn:SC2}) we have $|S_i|=O(|S|^2)$ and $|X_i|=O(|S|)$.
If $S(Z,A)$ is a quadratic system of equations, then the systems $S_i=S_i(X_i,A)$  are quadratic. \end{prop}

In order to prove the above proposition, we first prove the two following lemmas.

\begin{lemma}\label{interLem1}
Let  $S=S(Z,A)=1$ be a  system of equations over $\Gamma=\GammaPresentation$. We can rewrite $S$ as a system of triangular equations $S'$ such that $|S'|=O(|S|)$ (constants depend on $\Gamma$).  If $S$ is quadratic, then  $S'$ is also quadratic.
\end{lemma}

\begin{proof}
We can assume that $S$ consists of only one equation of the form $y_1 y_2 \cdots y_n=1$, where either $y_i \in Z$ or $y_i \in \Gamma$. The general case can be proved by a similar argument. At the first step of triangulation we introduce the new variable $x_1$ and we rewrite $S$ as

\[
 \begin{array}{c}
y_1y_2 x_1=1\\
x_1^{-1}y_3 \cdots y_n=1
\end{array} \]

If we continue the process we get a triangular system $S'$ of the following form:

\[
 \begin{array}{c}
y_1y_2 x_1=1\\
x_1^{-1}y_3x_2=1\\
x_2^{-1}y_4 x_3=1\\
\vdots \\
x_{n-3}^{-1} y_{n-1}y_n=1
\end{array} \]

The length of each triangular equation is bounded by $3$ and there are $(|S|-2)$ such equations. In addition we have to add the length of the coefficients, because some $y_i$ belong to $\Gamma$. Hence, $|S' | \leq 4|S|$.

If $S$ is quadratic, then there are at most two indices $i,j$ such that $y_i=y_j=z$. Hence each variable $z$ appears at most twice in $S'$.  Since each new variable $x_i$ also appears twice in $S'$, we conclude that  $S'$ is quadratic.
\end{proof}

\begin{lemma} \label{interLem2}
Let $S(A,Z)=1$ be a system of equations over $\Gamma$. If we assume that the system $S(A,Z)$ is  in a triangular form, then the systems $S_i$'s defined by equations  (\ref{Eqn:SC1})  and  (\ref{Eqn:SC2}) are of length  $|S_i|=O(|S|^2)$ and and $|X_i|=O(|S|)$ (constants depend on $\Gamma$).

If we assume that the system $S(A,Z)$ is quadratic and in a triangular form, then the systems $S_i$'s defined by equations (\ref{Eqn:SC1})  and  (\ref{Eqn:SC2}) are quadratic.

\end{lemma}

\begin{proof}
Fix a system $S_i$ defined by equations  (\ref{Eqn:SC1})  and  (\ref{Eqn:SC2}). First we observe that each equation in $S_i$ of the form $x_{k}^{(j)}c_{k}^{(j)}\left(x_{k+1}^{(j)}\right)^{-1} = x_{k'}^{(j')}c_{k'}^{(j')}\left(x_{k'+1}^{(j')}\right)^{-1} $ corresponds to  the  triples

\begin{eqnarray*}
\theta_{m} (g_{\sigma(j,1)}) & = & h_{1}^{(j)} c_{1}^{(j)} \left(h_{2}^{(j)}\right)^{-1} \label{CanonReps1}\\
\theta_{m} (g_{\sigma(j,2)}) & = & h_{2}^{(j)} c_{2}^{(j)} \left(h_{3}^{(j)}\right)^{-1}\\
\theta_{m} (g_{\sigma(j,3)}) & = & h_{3}^{(j)} c_{3}^{(j)} \left(h_{1}^{(j)}\right)^{-1}.\label{CanonReps3}
\end{eqnarray*}

and

\begin{eqnarray*}
\theta_{m} (g_{\sigma(j',1)}) & = & h_{1}^{(j')} c_{1}^{(j')} \left(h_{2}^{(j')}\right)^{-1} \label{CanonReps1}\\
\theta_{m} (g_{\sigma(j',2)}) & = & h_{2}^{(j')} c_{2}^{(j')} \left(h_{3}^{(j')}\right)^{-1}\\
\theta_{m} (g_{\sigma(j',3)}) & = & h_{3}^{(j')} c_{3}^{(j')} \left(h_{1}^{(j')}\right)^{-1}.\label{CanonReps3}
\end{eqnarray*}

where $\sigma(j,k)=\sigma(j',k')$. If $S$ is quadratic, for each pair  $(j,k)$, there is at most one pair $(j',k')$ such that $\sigma(j,k)=\sigma(j',k')$. Hence, there is at most one equation  $x_{k}^{(j)}c_{k}^{(j)}\left(x_{k+1}^{(j)}\right)^{-1} = x_{k'}^{(j')}c_{k'}^{(j')}\left(x_{k'+1}^{(j')}\right)^{-1} $, involving $x_{k}^{(j)}c_{k}^{(j)}\left(x_{k+1}^{(j)}\right)^{-1}$. We conclude that $x_{k}^{(j)}$ appears at most in two equations of the following form:

\begin{eqnarray*}
x_{k}^{(j)}c_{k}^{(j)}\left(x_{k+1}^{(j)}\right)^{-1} & = & x_{k'}^{(j')}c_{k'}^{(j')}\left(x_{k'+1}^{(j')}\right)^{-1} \\
x_{k-1}^{(j)}c_{k-1}^{(j)}\left(x_{k}^{(j)}\right)^{-1} & = & x_{k''}^{(j'')}c_{k''}^{(j'')}\left(x_{k''+1}^{(j'')}\right)^{-1}.
\end{eqnarray*}
Hence, $S_i$ is quadratic.


In order to find an upper bound on the length of $S_i$ we have to look at  two types of equation which appear in  $S_i$. The first type  corresponds to variables in $S$. For each occurrence of variable $z_{\sigma(j,k)} \in S$ we have an equation  $s: x_{k}^{(j)}c_{k}^{(j)}\left(x_{k+1}^{(j)}\right)^{-1} = x_{k'}^{(j')}c_{k'}^{(j')}\left(x_{k'+1}^{(j')}\right)^{-1} $. The length of such equation is bounded by $4+(|c_{k}|+|c_{k'}|)$. We have that  $|c_{k}|+|c_{k'}|$ is bounded by $2L$, where $L$ is the constant introduced in the definition of canonical representatives in the proof of Proposition \ref{Lem:RipsSela1}. Hence, we get

\begin{equation*}
|s| \leq 4+2L
\end{equation*}

The second type of equations corresponds to constants appearing in $S$. For each constant $a_s \in S$, there is an equation  $x_{k}^{(j)}c_{k}^{(j)}\left(x_{k+1}^{(j)}\right)^{-1} = \theta_{m}(a_s)$ in $S_i$, where $\theta_{m}(a_s) \in F_A$ is the label of a $(\lambda, \mu)$-quasi-geodesic path from $1$ to $a_s$ in the Cayley graph of  $\Gamma$, for some $\lambda$ and $\mu$ depending only on $\Gamma$ \cite{RS95}. Thus, $|\theta_{m}(a_s)| \leq \lambda|a|+ \mu \leq \lambda|S|+ \mu$. Therefore $2+L+  \lambda|S|+ \mu$ is an upper bound for the length of such equation.

Since there are at most $|S|$ equations in $S_i$, we get $|S_i| \leq |S|(4+2L+\lambda|S|+\mu)=O(|S|^2)$. This finishes the proof of  the lemma.
\end{proof}

Now we can prove Proposition \ref{Bnd}.

\begin{proof}(of Proposition \ref{Bnd})
 We rewrite $S(Z,A)$ as a triangular system $S'(Z',A)$ and we consider the systems $S'_i$'s defined by equations (\ref{Eqn:SC1})  and  (\ref{Eqn:SC2}) for $S'$. From Lemma \ref{interLem1} and Lemma \ref{interLem2} , we have $|S'_i|=O(|S'|^2)=O(|S|^2)$ and and $|X_i|=O(|S|)$ which proves Proposition \ref{Bnd}.

\end{proof}

We recall Theorem 1.1 in \cite{lm}:
if a quadratic equation $Q(X,A) =1$ has a solution  in a free group, then there exists a solution $\phi$  such that for each variable $x\in X$, $|\phi (x)|\leq 40n(Q)c(Q)$
for orientable equation and $|\phi (x)|\leq 150n(Q)c^2(Q)$ for non-orientable equation.

This theorem and Proposition \ref{Bnd}   imply the statement of Theorem \ref{th:1}.

Proposition \ref{Bnd}   and the result of Gutierres \cite{Gu} about the PSpace algorithm for solving equations in a free group imply the following result.
\begin{theorem}  \cite{kufl} Let $\Gamma$ be a   torsion-free hyperbolic group. There is a  PSpace algorithm for  solving equations in $\Gamma$. \end{theorem}

\section{Quadratic equations in free products}

In this section we will prove results about quadratic equations in free
 products, which will be used for the proof of Theorem \ref{th:2} in the next section. Suppose an element $h$ in a free product of a free group and free abelian groups is written in canonical form as $h=h_1\ldots h_k$. Then the length of $h$ is defined as the sum of the lengths of $h_1,\ldots, h_k$.

\begin{theorem}  \label{th:5}   Let $Q$ be a quadratic word. If the equation $Q = 1$
 is solvable in a free product of a free group and free abelian groups
 of finite ranks, then there exists a solution $\alpha$ such that for
 any variable $x$,
$|\alpha (x)|< Nn(Q)((n(Q)+c(Q))^2$ if $Q=1$ is orientable, and
 $|\alpha (x)|< Nn^2(Q)((n(Q)+c(Q))^5$ if $Q=1$ is non-orientable.
 Here $n(Q)$ denotes the total number of variables in $Q$ and $c(Q)$
 the total length of coefficients occurring in $Q.$  One can take $N=400$ for orientable equation and $N=9000$ for non-orientable equation.\end{theorem}

In this section we use topological and graph theoretic methods to work
with quadratic equations in free groups and free products of groups,
introduced in \cite{C}, \cite{Olshanskii-1989}, \cite{Comerford-Edmunds-1981},
\cite{CE}, \cite{V,V2,V5}. 

While all necessary definitions are given here, further relevant background on quadratic equations in free groups and free products may be found in the works mentioned above. Here, we do not distinguish edges of graphs from their labels, in order to avoid complicated notation.

\begin{definition} 
We define the \emph{orientable genus}  of an $m$-tuple
$\{C_i,C\} \subseteq G$, denoted by $Genus(C_1,C_2,\ldots,C)$,
 to be the least integer $g\geq 1$ for
which the equality \bel{eqn:orientable} \left( \prod_{i=1}^g [x_i,y_i]
\right)\left(\prod_{j=1}^{m-1} z_j^\mo C_j z_j\right) C = 1 \eee  holds for some $\{x_i,y_i,z_j\} \subseteq G$, where $[x,y]=x^\mo y^\mo
xy$. We say, that the orientable genus of $\{C_i,C\} \subseteq G$ is 0
if  there exists a tuple $\{z_1,\ldots ,z_{m-1}\}$ such that
$$(\prod_{j=1}^{m-1} z_j^\mo C_j z_j) C = 1.$$ If such integer $g$ does not exist, the genus is not defined.
\end{definition}

\begin{definition} For a group $G$, 
we define the \emph{non-orientable genus}  of an $m$-tuple
$\{C_i,C\} \subseteq G$, denoted by $Sq(C_1,C_2,\ldots,C)$,
 to be the least integer $g\geq 1$ for
which the equality \bel{eqn:non-orientable}
 \left(\prod_{i=1}^g x_i^2\right)
\left(\prod_{j=1}^{m-1} z_j^\mo C_j z_j\right) C = 1 \eee holds
for some  $\{x_i,z_j\} \subseteq G$.
We say, that the non-orientable genus of $\{C_i,C\} \subseteq G$ is 0
if  there exists a tuple $\{z_1,\ldots ,z_{m-1}\}$ such that
$$(\prod_{j=1}^{m-1} z_j^\mo C_j z_j) C = 1.$$ If such integer $g$ does not exist, the genus is not defined.
\end{definition}

\begin{definition}
 An \emph{orientable quadratic set of words}  is a quadratic set
 of cyclic  words
$ w_1,w_2,\ldots ,w_{k}$
 (a cyclic word is the orbit of a linear word under cyclic permutations)
 in some alphabet $a_1^{\pm 1},a_2^{\pm 1},\ldots$,  such that every
 $a_i$ appears a total of twice in $w_1,w_2,\ldots,w_k$, once as $a_i$ and once as $a_i^{-1}$.
\end{definition}

\begin{definition}

 A \emph{non-orientable quadratic set of words} is a quadratic
 set of cyclic  words
$ w_1,w_2,\ldots ,w_{k}$ in some alphabet $a_1^{\pm 1},a_2^{\pm 1},\ldots$, such that every
  $a_i^{\pm 1}$ appears a total of twice in $w_1,w_2,\ldots,w_k$, and there is at least one $i$ such that $a_i^{\pm 1}$ appears twice with the same exponent.
\end{definition}

Similar to the case of Wicks forms (see \cite{V},\cite{kv} for background on Wicks forms), we associate a union of surfaces to a quadratic set of words. If the quadratic set is orientable, then all surfaces are orientable, whereas if the quadratic set is non-orientable then at least one surface is non-orientable.  This is done by taking a set of disks with boundaries labeled by words from the quadratic set, and identifying edges which have the same labels, respecting orientation.
The {\it genus}  of a quadratic set of words is defined as the sum 
of genera of the surfaces obtained from $k$ 
disks with words $ w_1,w_2,\dots ,w_{k}$ on their boundaries.
We will denote the genus by $Genus( w_1,w_2,\dots ,w_{k})$ in the orientable case, and by $Sq( w_1,w_2,\dots ,w_{k})$ in the non-orientable case.
If a set of words is not strictly quadratic, and each letter appears an even number of times, topological genus may be defined.

\begin{definition}
Let elements $\{C_i\}$ be represented in $G$ as a specialization  
of a quadratic set of words (not necessarily uniquely). 
The minimum of topological genera of such quadratic sets of words 
is the \emph{topological genus} of the tuple $\{C_i\}$, $i=1 \ldots m$.

\end{definition}
 The genus given in Definitions 1 and 2 is called \emph{algebraic}.
 
 \begin{lemma}
  
  The algebraic genus of a tuple $\{C_i, \i=1 \ldots m\}\subseteq G$  is equal to the
  topological one.
 \end{lemma}
 
 {\bf Proof.} 
 Let $g$ be the algebraic genus of a tuple $\{C_i, \i=1 \ldots m\}\subseteq G$. Then $\{C_i, i=1 \ldots m\}$ can be presented as a specialization
of a quadratic set of words of genus $g$. As the quadratic set of words
we may take a set $\prod_{i=1}^g [x_i,y_i]t_1,t_1^{-1}t_2,\dotsc,t_{m-2}^{-1}t_{m-1},t_{m-1}^{-1}$, for example. Using the Euler characteristic formula,
we see that the genus of this quadratic set is $g$.
Therefore, the algebraic genus is greater than or equal to
the topological one. 

Now let the topological genus of a tuple
$\{C_i, i=1 \ldots m\}$, be $g$. Then
there is a quadratic
set of words $U_i$, $i=1 \ldots m$, of genus $g$ such that
a system $C_i=U_i$, $i=1 \ldots m$, has a solution in $G$. Consider
first the case
when the quadratic set of words $U_i$, $i=1 \ldots m$, defines one surface of genus $g$.
Then the system 
$C_i=U_i$, $i=1 \ldots m$, is equivalent to one quadratic equation. Indeed, we can express a letter that occurs only once in $U_m$ from the equation $C_m=U_m$, and substitute it into the remaining system, then continue eliminating letters until we get one equation. Then the algebraic genus
is less than or equal to $g$, so the algebraic and topological genera coincide.

If the quadratic set defines several surfaces, then the equalities
(8) and (9) can be presented as several independent ones of the same form. We complete the proof by induction, using the above arguments for each connected component.

The lemma is proved.

\bigskip

 Consider the orientable (non-orientable) compact surface $S$ associated to an orientable (non-orientable)
 quadratic set. This surface has an embedded graph
 $\Gamma\subset S$ such that $S\setminus \Gamma$ is a set of open polygons.
 This construction also works
 in the opposite direction. Given a graph $\Gamma\subset S$
 with $e$ edges on an orientable (non-orientable) compact connected surface $S$ of genus $g$ such that $S\setminus \Gamma$ is a collection of disks, by labeling and orienting the edges of $\Gamma$, and cutting $S$ open along $\Gamma$, we get an orientable (non-orientable) quadratic set of words of genus $g$. The associated orientable (non-orientable) quadratic set of words can be read on the boundary of the resulting polygons.     
 We henceforth identify an orientable (non-orientable)
 quadratic set with the associated embedded graph $\Gamma\subset S$,
 allowing the language of vertices and edges of orientable (non-orientable)
 quadratic sets to be used. Moreover, the quadratic set can be associated with 
a set $T$ of closed
paths in $\Gamma$, where every edge of $\Gamma$ is traversed exactly twice.
 We will call
$T$ a {\em quadratic set of circuits.} If the genus of the surface $S$ is $g$,
then the genus of $T$ is $g$, in the sense of Definitions 1 and 2 (see \cite{C},\cite{CE}, for
example).

\begin{definition}
Let $v$ be a vertex of $\Gamma$ with edges $a_1,...,a_l$ originating from $v$.
If $\Gamma$ corresponds to an orientable quadratic set $Q$, then this set
contains subwords $a_1^{-1}a_2,\ldots,a_{l-1}^{-1}a_l,a_l^{-1}a_1$ (which may be contained in different words of the set).
If $Q$ is non-orientable, then some of these subwords may be reversed.
This set of subwords will be called \emph{girth} of the vertex $v$.
We  say that the vertex $v$ is \emph{extended} by a 
 word $W$ in generators
of some group $H$, $\psi_1\ldots\psi_l=W$ in $H$, $\psi_i \in H$, 
if subwords $a_1^{-1}a_2,\ldots,a_{l-1}^{-1}a_l,a_l^{-1}a_1$  are replaced
by a new set of words
 $a_1^{-1} \psi_1 a_2,\ldots,a_{l-1}^{-1} \psi_{l-1} a_l,a_l^{-1} \psi_l a_1$.
 If a subword $a_j^{-1}a_{j+1}$
is reversed, i.e. occurs as $a_{j+1}^{-1}a_j$, then the corresponding $\psi$
appears in the product with negative exponent.

\end{definition}



\begin{example}
Consider the following non-orientable genus 2 word \newline
$AB^{-1}AC^{-1}BC^{-1}$. The corresponding graph $\Gamma$ consists
of three multiple edges $A,B,C$ originating in a vertex $v$ and terminating
in a vertex $u$ (see Figure 1).
We say that a vertex $v$ is extended by a word $W$ if the word
$AB^{-1}AC^{-1}BC^{-1}$ is replaced by a word
$AB^{-1}\psi_3AC^{-1}\psi_2BC^{-1}\psi_1$, such that
 $\psi_1^{-1}\psi_2\psi_3=W \in H$.

\end{example}
\begin{definition}\label{genusExtension}

Let $Q$ be a 
 quadratic set, and $\Gamma$
be the associated graph.
Let vertices $v_1,\ldots,v_t$ of $\Gamma$ be extended by words
 $W_1,\ldots,W_t \in H$.
We will say, that an \emph{orientable (non-orientable) genus $g$ joint extension
(or $g$-extension)}
$\Delta$ of $Q$ is constructed on these $t$ vertices, if

\begin{itemize}
\item  the $t$-tuple $(W_1,...,W_t) $ is orientable and the orientable
 genus of $(W_1,\ldots,W_t) $
is $l=g-t+1$ in $H$ or 
\item the $t$-tuple $(W_1,\ldots,W_t) $ is  non-orientable and the non-orientable
 genus of $(W_1,\ldots,W_t) $
is $l=g-2t+2$ in $H$. 
\end{itemize}
 
 The sum of the lengths of $W_i$ will be called the length
 of the extension.
\end{definition}
\begin{figure}[ht!]
\centering
\includegraphics[width=40mm]{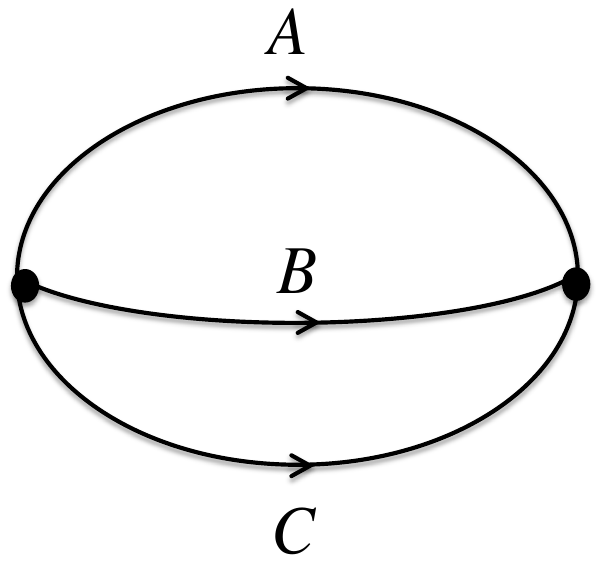}
\caption{ }
\label{bundle}
\end{figure}
\label{example1}
 \begin{definition} \label{multiform}

Let $G$ be a free product (finite or infinite) of groups $G_1,G_2,\ldots$.
Let $Q$ be an orientable (non-orientable) quadratic set of genus $k$ and
 $\Gamma$
be the associated graph. We  separate all vertices of $\Gamma$
into $p+1$ disjoint sets $B_0,B_1,\ldots, B_p$.  Leave the  vertices 
of $B_0$
 without changes. The genus $g_i$ orientable, or non-orientable, joint
 extension by a word of some free factor is constructed
on the vertices of the set $B_i$, for $i=1,\ldots,p$ (different sets can be
extended by words of the same free factor, or of different free factors). 
We consider the following cases:

\begin{enumerate}

\item  \label{case 2} $Q$ is non-orientable. Then we define $n$ to be 
  \begin{equation*}
  n=k+\sum_{  i :g_i-extension  \text{is non-orientable }} g_i +\sum_{i:g_i-extension \text{ is orientable }} 2g_i
  \end{equation*}

\item  \label{case 3} $Q$ is orientable and at least one of the $g_i$-extensions is non-orientable. Then we define $n$ to be 
  \begin{equation*}
  n=2k+\sum_{  i :g_i-extension  \text{is non-orientable }} g_i +\sum_{i:g_i-extension \text{ is orientable }} 2g_i
  \end{equation*}
  
 \item  \label{case 4} $Q$ is orientable and all $g_i$-extensions are orientable. Then we define $n$ to be 
  \begin{equation*}
  n=k+\sum_{i=1}^p g_i 
  \end{equation*}
\end{enumerate} 
If the quadratic set and
 all the extensions are orientable, we call the resulting set of words an
 \emph{orientable multi-form of genus $n$}. If the quadratic set is non-orientable, or at least one of the extensions if non-orientable, then we call the resulting set of words a \emph{non-orientable multi-form of genus $n$}. Denote the multiform by $\mathcal{A}$.

\end{definition}

\begin{definition}

The set $Q$ from Definition \ref{multiform}
 will be called a \emph{framing set of words of the multi-form $\mathcal{A}$}.

\end{definition}

\begin{definition}
Let a quadratic set $Q$ be a collection of words in a group
alphabet $\mathcal{B}$, and $\mathcal{A}$ be a multi-form such that $Q$ is the framing set of words for this multi-form. Let $V_1,\ldots ,V_l$ be a set of elements of a free product $G$, which is obtained from the multi-form $\mathcal{A}$ over the free product, by substitution of letters of the alphabet $\mathcal{B}$ with elements of $G$. We will say that the family of words $V_1,\ldots , V_l$ is obtained from $\mathcal{A}$
by a \emph{permissible substitution}, if elements of the same free factor don't occur in the words $V_1,\ldots , V_l$ side
by side.

\end{definition}

\begin{example}
We give an example of a non-orientable multi-form of non-orientable genus $13$
(in the sense of Definition 7):
$$V_1=A\xi_1ED\psi_1C^{-1}B\xi_3A^{-1}FB\xi_2EG_1G_2H_1O_2O_1^{-1}I_1,$$
$$V_2=C\psi_2H_2O_2ZG_2I_2I_1F,$$
$$V_3=D\psi_3H_2H_1^{-1}I_2O_1ZG_1,$$
where $\xi_1\xi_2^{-1}\xi_3=U_1$, $\psi_1\psi_2\psi_3^{-1}=U_2$, and the orientable genus of $(U_1,U_2)$ in some free factor $G_i$ is 3, where $G_i$ is a free group.

If we write the words $V_1,V_2$ and $V_3$ around $3$ disks (ignoring $\psi_i$'s and $\xi_i$'s) and identify edges according to their labels, we get a non-orientable surface of genus $5$. The graph $\Gamma$ is shown in Figure \ref{fig:ex2}. $\Gamma$ has $9$ vertices, $15$ edges, and  $3$ faces, because there were $3$ disks. The vertices of $\Gamma$ are marked with numbers from
1 to 9 , see Figure \ref{fig:ex2}. All vertices of $\Gamma$ are separated in two sets, $B_0$ and $B_1$, where $B_0=\{1,3,4,5,6,7,9\}$
and $B_1=\{2,8\}$, where there is an extension performed on the vertices of
$B_1$. Now we use the formula $v-e+f=2-g$. The genus of the extension obtained by $U_1$ and $U_2$ is $3+(2-1)=4$, since $(U_1,U_2)$ is an orientable set (see Definition \ref{genusExtension}). Now by Definition \ref{multiform}, the genus of the multiform is $5+2(4)=13$ (since $V_1,V_2$ and $V_3$ define a non-orientable surface and $(U_1,U_2)$ is orientable,
 we use case \ref{case 2} in Definition \ref{multiform}).

\end{example}
\begin{figure}[ht!]
\labellist
\hair 2pt

\pinlabel $A$ at 67 130
\pinlabel $B$ at 86 76
\pinlabel $C$ at 210 28
\pinlabel $D$ at 263 88
\pinlabel $E$ at 145 88
\pinlabel $F$ at 26 113
\pinlabel $G_1$ at 248 139
\pinlabel $G_2$ at 210 155
\pinlabel $H_1$ at 258 198
\pinlabel $H_2$ at 306 144
\pinlabel $I_1$ at 112 200
\pinlabel $I_2$ at 186 217
\pinlabel $O_1$ at 146 178
\pinlabel $O_2$ at 240 171
\pinlabel $Z$ at 172 121
\pinlabel $1$ at 70 178
\pinlabel $2$ at 68 99
\pinlabel $3$ at 70 40
\pinlabel $4$ at 210 90
\pinlabel $5$ at 150 160
\pinlabel $6$ at 230 211
\pinlabel $7$ at 286 174
\pinlabel $8$ at 303 99
\pinlabel $9$ at 158 213

\endlabellist
\centering
\includegraphics[scale=1]{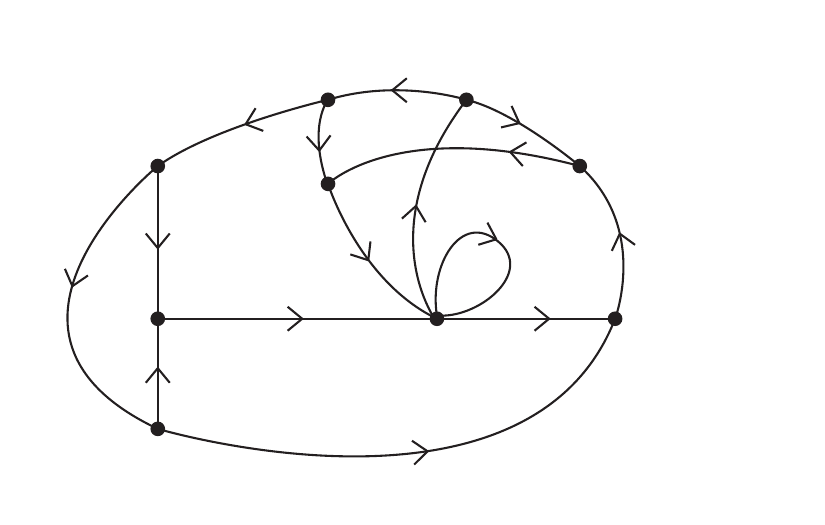}
\caption{}
\label{fig:ex2}
\end{figure}

 As an example of case \ref{case 3}, one  can consider
 $A\psi_1B\psi_4 A^{-1}\psi_3B^{-1}\psi_2,$ where $\psi_1\psi_2\psi_3\psi_4$ is a square in some free factor.
This is a non-orientable multi-form of genus 3.

\begin{lemma} \label{genuslemma}
Let $\Delta$ be a connected genus $k$ surface with $t$ holes, and with a quadratic set of words $u_1,\ldots, u_t$ written on the boundaries of those holes. Let
$\Delta'$ be the genus $n$ surface obtained by identification of
the holes in $\Delta$ according to their labels. Let  $\Delta_1$ be
the compact closed surface obtained from $t$ disks with $u_1,\ldots,u_t$
written on their boundaries, by identifying the 1-cells according to their labels. Then we have the following cases.

\begin{enumerate}[(i)]
\item If $\Delta$ and the set $(u_1,\ldots,u_t)$ are both orientable, then
   \begin{equation*}
   Genus(u_1,\ldots, u_t)=Genus\Delta_1=n-k-t+1.
   \end{equation*}
   
 \item \label{nonorientableCase} If $\Delta$ and the set $(u_1,\ldots,u_t)$ are both non-orientable, then 
    \begin{equation*}
   Sq(u_1, \ldots, u_t)=Sq\Delta_1=n-k-2t+2.
   \end{equation*}
   
\item If $\Delta$ is orientable and the set $(u_1,\ldots,u_t)$ is non-orientable, then
   \begin{equation*}
   Sq(u_1, \ldots, u_t)=Sq\Delta_1=n-2k-2t+2.
   \end{equation*}

\item If $\Delta$ is non-orientable and the set $(u_1,\ldots,u_t)$ is  orientable, then
    \begin{equation*}
   Sq(u_1, \ldots, u_t)=Sq\Delta_1=\frac{n-k-2t+2}{2}.
   \end{equation*}
\end{enumerate}
\end{lemma}

\begin{proof}
We give a proof for the case \eqref{nonorientableCase}. The other cases can be proved in a similar way.  Let $\Delta_1$ be the surface obtained by writing the words $u_1,\ldots,u_t$ around the boundaries of $t$ disks and identifying the 1-cells according to their labels, as in the statement of the Lemma. Then the surface $\Delta'$ is a connected sum of $\Delta$, $\Delta_1$, and $t-1$ tori:

\begin{equation*}
\Delta'=\Delta \# \Delta_1 \# \underbrace{T \# \dotsb\# T}_\textrm{$(t-1)$ times}.
\end{equation*}

Since $\Delta $ and $\Delta_1$ are both non-orientable, this connected sum is the same as a connected sum of $k+Sq{\Delta_1}$ projective planes and $t-1$ tori. Hence, 

\begin{equation*}
\Delta'= \underbrace{P \# \cdots \#P}_\textrm{$(k+g_{\Delta_1})$ times} \# \underbrace{T \# \cdots \# T}_\textrm{$(t-1)$ times}.
\end{equation*}

Since the connected sum of a torus and a projective plane is homeomorphic to the connected sum of three projective planes, we have:

\begin{equation*}
\Delta'= \underbrace{P \# \cdots \#P}_\textrm{$(k+Sq{\Delta_1}+2)$\small{ times}} \# \underbrace{T \# \cdots \# T}_\textrm{$(t-2)$ times}.
\end{equation*}

Continuing to rewrite in terms of connected sums of projective planes, we obtain:

\begin{equation*}
\Delta'= \underbrace{P \# \cdots \#P}_\textrm{$(k+Sq{\Delta_1}+2(t-1))$\small{ times}}.
\end{equation*}

Hence, the genus of $\Delta'$ is equal to $k+Sq{\Delta_1}+2(t-1)$:

\begin{equation*}
n=k+Sq{\Delta_1}+2(t-1).
\end{equation*}

Solving this for $Sq{\Delta_1}$, we get

\begin{equation*}
Sq{\Delta_1}=  Sq(u_1,\ldots,u_t)=n-k-2(t-1),
\end{equation*}

which proves \eqref{nonorientableCase}.
\end{proof}

Let $G$ be a free product of groups $G_1, G_2,\ldots$. Suppose a standard orientable quadratic equation has a solution $\phi : F(x_i,y_i,z_j)*G\rightarrow G$. Represent the equation in the form such that the product of commutators is equal to the product of conjugates of the coefficients:
\bel{oeq} \left( \prod_{i=1}^g [x_i,y_i]
\right)=\left(\prod_{j=1}^{m-1} z_j^\mo V_j z_j\right) V_m \eee 
The product of commutators is a quadratic word. The solution $\phi $ defines a specialization of this quadratic word  in $G$. This specialization  can be represented as a product of several elements $U_1,\ldots,U_m$, where each  $U_i$ is a conjugate of  $V_i$. Therefore,  the set of elements $U_1,\ldots,U_m$
   can be obtained (altogether) as a specialization $\phi$ of a quadratic set of words.  
\begin{prop}\label{th:6} In the above notation,
elements $V_1,\ldots, V_m$
can be obtained from a multi-form of genus $g$ over $G$ by a permissible substitution.
\end{prop}
\begin{proof}
The orientable quadratic set of words, and the specialization $U_1$,\ldots,$U_m$, can be written on the boundaries of $m$ disks labeled by  elements of the free product. We do not make any cancellations at this point between the specializations of the letters in these words. The disks define a surface of genus $g$ (we can assume that they define one connected surface, otherwise this specialization corresponds to a solution of two disjoint quadratic equations, where the sum of their genera is $g$). The boundaries of disks give a graph $\Gamma$ on the surface, with elements of the free product written on the edges. Each element labeling an edge is in normal form for the free product. We divide the edges of $\Gamma$ with vertices of degree two, according to the normal forms of the words written on the edges. We say that there is a labeling function $\phi$ on the edges of $\Gamma$ by elements of the free product, and $\Gamma$ may be equipped with a quadratic set of circuits $T$ such that the words representing $U_1,\ldots ,U_m$ can be read along $T$.

Let $p$ be a path of length $k$ in $\Gamma$ such that $\phi(p)=1$, and let $v$ and $w$ be its end points. Vertices $v$ and $w$ do not coincide, since otherwise the genus of the equation would be decreased, by the Euler characteristic formula. There is an edge $e_j$ which appears in $p$ exactly once (otherwise the genus of the equation would be smaller, again by the Euler characteristic formula). We identify the vertices $v$ and $w$ of $\Gamma$ and delete $e_j$. Then we delete all the edges which are incident to at least one vertex of degree one.

We say that the new graph $\Gamma^{\prime}$ is obtained from
$\Gamma$ by $\tau_k$-transformation. Since the number of edges in  $\Gamma^{\prime}$ is at least one less than the number of edges in $\Gamma$, if this process is continued then we get a graph $\Gamma^{\prime}$ and quadratic set of circuits $T^{\prime}$. The words $U_1$,...,$U_m$ can be read along the circuits $T^{\prime}$, and  $\Gamma^{\prime}$ does not have any subpath $p$ such that $\phi(p)=1$.

In the next step, we show that  removing all connected components of $\Gamma^{\prime}$ which are labeled by elements of the same free factor,
corresponds to constructing joint extensions of a framing word.

Let $K$ be a connected subgraph of $\prm{\Gamma}$ whose edges are labeled by elements of a free factor
$G_i$. We take  $K$ to be maximal. We call a vertex a {\it boundary } vertex  if it is incident to both
$K$ and $\prm{\Gamma} \setminus K$. We refer to the edges of $K$ as  $K$-edges, and to the edges of $\prm{\Gamma} \setminus K$ as $\Complement{K}$-edges.

Let $w$ be a boundary vertex. Without loss of generality, we assume that all edges incident to $w$ are leaving $w$. Let

\begin{align*}
a^{-1}_{1,1} a_{1,2}, \dotsc, a^{-1}_{1,t_{1}-1} a_{1,t_{1}},a^{-1}_{1,t_{1}}  b_{1,1}, b^{-1}_{1,1} b_{1,2},
&\dotsc, b^{-1}_{1,s_1} a_{2,1}, \\
& \dotsc,  a^{-1}_{2,t_{2}} b_{2,1} ,\dotsc,  a^{-1}_{r,t_{r}} b_{r,1}, \dotsc,  b^{-1}_{r,s_{r}} a_{1,1}
\end{align*}

\noindent be a girth of $w$, where $b_{i,j}$ edges are $K$-edges and $a_{i,j}$ edges are $\Complement{K}$ edges. We refer to the sets of $K$-edges ($\Complement{K}$-edges ) not separated by $\Complement{K}$-edges ($K$-edges)
as $K$-bundles ($\Complement{K}$-bundles) (see Figure \ref{bundle}).

\begin{figure}[ht!]
\centering
\includegraphics[width=110mm]{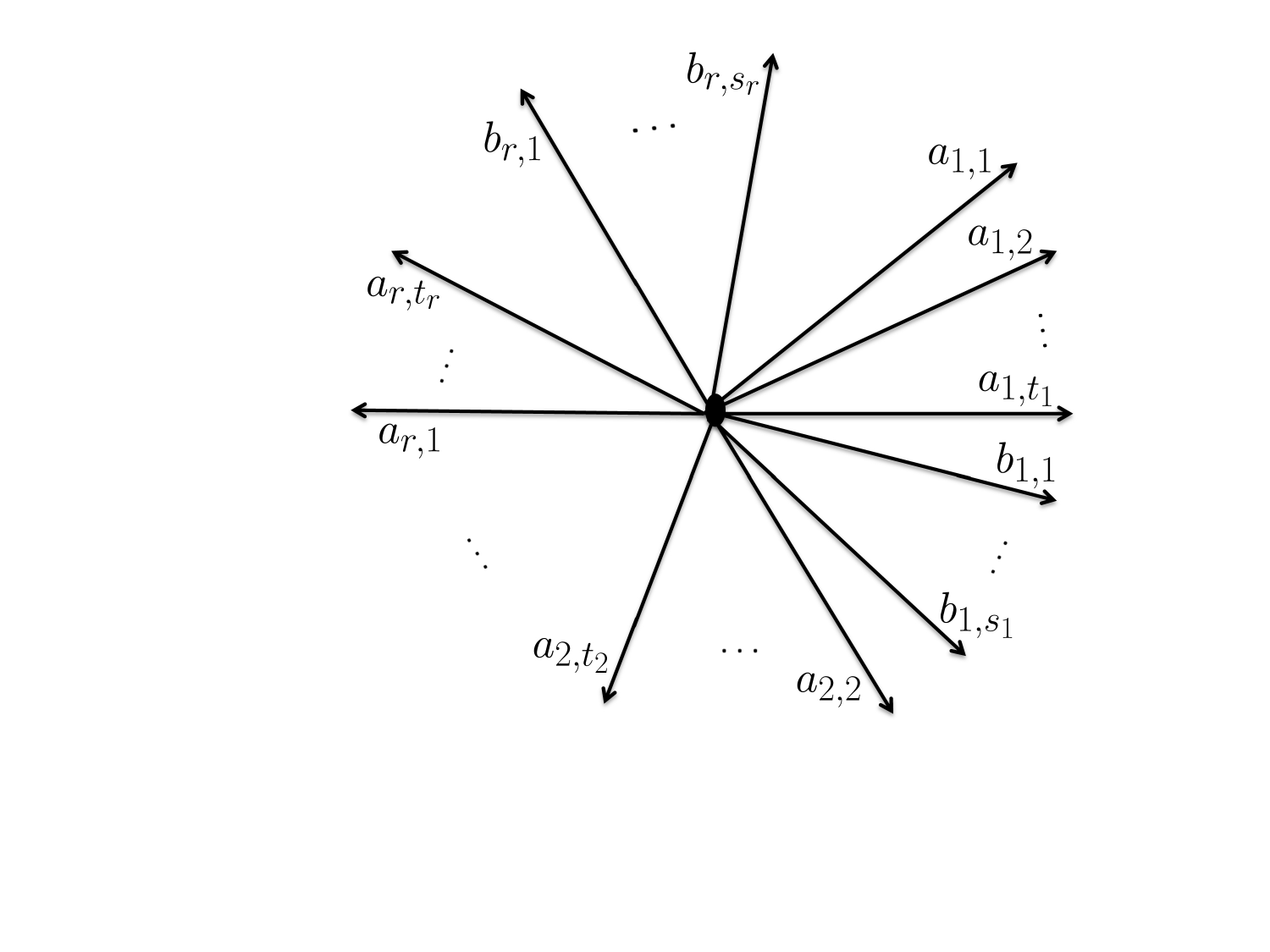}
\caption{ }
\label{bundle}
\end{figure}

We replace the vertex $w$ by a sequence of vertices $w_{1,1}, w_{1,2},\ldots,  w_{r,1}, w_{r,2}$, and insert an edge $e_i$ between $w_{i,1}$ and $w_{i,2}$ pointing towards $w_{i,2}$ and
an edge $\prm{e_i}$ between $w_{i,2}$ and $w_{i+1,1}$ pointing towards $w_{i,2}$. Each $\Complement{K}$-bundle $a_{i,1},\ldots,a_{i,t_{i}}$ is incident to $w_{i,1}$, and
each $K$-bundle  $b_{i,1},\ldots,b_{i,s_{i}}$ is incident to $w_{i,2}$. We let $\phi(e_i)=\phi(\prm{e_i})=1$ and consider them as $\Complement{K}$-edges. We call $w_{i,1}$ a boundary $\Complement{K}$-vertex (see Figure \ref{separation}). If we do this for all boundary vertices, we get a new graph
 $\Gamma''$ with a quadratic set of circuits $T''$ such that the words $U_1,\ldots,U_m$  can be read along $T''$.
A subpath of $T''$ consisting of only $K$($\Complement{K}$)-edges is called $K$($\Complement{K})$-subpath.

\begin{figure}[ht!]
\centering
\includegraphics[width=110mm]{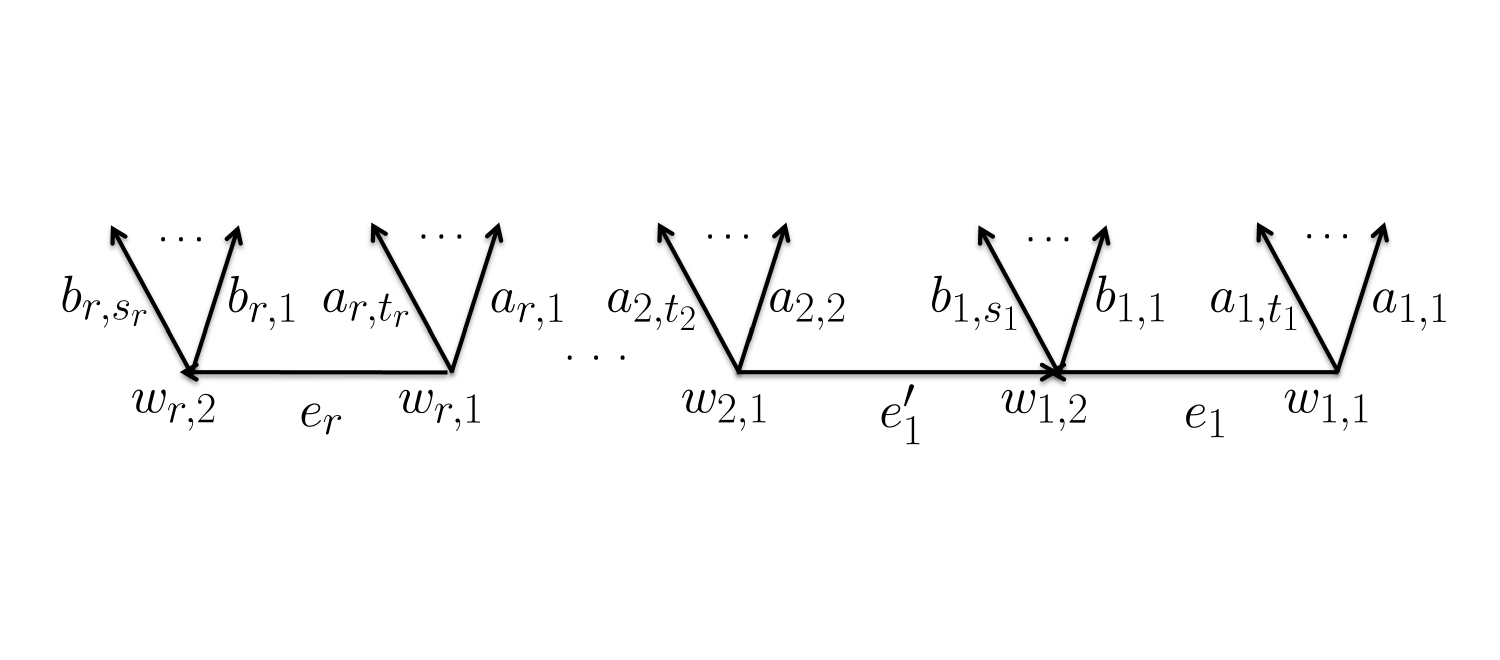}
\caption{ }
\label{separation}
\end{figure}

Let $v_1=w_{1,1}$ be a boundary $\Complement{K}$-vertex with $a_{1,1},\dotsc,a_{1,t_{1}}$ and $e_1$ leaving $v_1$. Then the girth of $v_1$ is

\begin{equation*}
e_{1}^{-1} a_{1,1}, a^{-1}_{1,1} a_{1,2}, \dotsc, a^{-1}_{1,t_{1}-1} a_{1,t_{1}}, a^{-1}_{1,t_{1}}e_1.
\end{equation*}

Let $B_1$ be the $K$-subpath which is traversed after $e_1$ in $T''$. Let $e_{2}^{-1}$ be the first edge
which is taken after this $K$-path by $T''$, and $v_2$ be the initial vertex of $e_2$. Hence, we have
the following sequence of subpaths in $T''$:

\begin{align*}
e_{1}^{-1} a_{1,1}, a^{-1}_{1,1} a_{1,2}, \dotsc, a^{-1}_{1,t_{1}-1} a_{1,t_{1}}, a^{-1}_{1,t_{1}}e_1
& B_1e_{2}^{-1} a_{2,1}, a^{-1}_{2,1} a_{2,2}, \\
& \dotsc, a^{-1}_{2,t_{2}-1} a_{2,t_{2}}, a^{-1}_{2,t_{2}}e_2.
\end{align*}.

A subpath $B_2$ follows $e_2$ in $T''$, and so on. Since $T''$ is finite,
 at some point it comes back to $v_1$. So there exists a sequence of subpaths of $T''$ of the following form:

\begin{equation}\label{chainExample}
\begin{split}
a^{-1}_{1,1} a_{1,2}, \dotsc, a^{-1}_{1,t_{1}}e_1B_1e_{2}^{-1} a_{2,1},
& \dotsc, a^{-1}_{2,t_{2}}e_2B_2e_{3}^{-1} a_{3,1},\\
& \dotsc, e_{m}^{-1} a_{m,1},\dotsc,
 a^{-1}_{m,t_{m}}e_mB_m e_{1}^{-1}  a_{1,1} .
 \end{split}
\end{equation}

We will call such a sequence a \emph{chain}. Each vertex is included in some chain, and boundary $\Complement{K}$-vertices can be partitioned into disjoint classes according to the chain in which they appear. Let $t$ be the number of these classes.

If we write $T''$ along the boundary of several disks, and identify
 $\Complement{K}$-edges and edges $e_{ij}$
according to their labels, we get a surface $\Delta$ with $t$ boundary components $\Delta_1,\dotsc,\Delta_t$. The labels of these boundary components are $K$-paths. We denote these labels by $u_i$. For example, if we consider
 the chain \eqref{chainExample} mentioned above, the cyclic word $u=B_1 \cdots B_m$ will be written
around one of these boundary components.  Let the genus of $\Delta$ be $k$.

We glue a disk $D_i$ to each boundary component $\Delta_i$ by identifying the boundary of $D_i$ with the boundary of $\Delta_i$. Then we shrink each $D_i$ to a vertex $v_i$. Let $\Delta'''$ be the surface that we get after this operation. Then $\Delta'''$ has a graph $\Gamma'''$ and a family of  circuits $T'''$ such that the words $U'''_1,\dotsc,U'''_m$  can be read along $T'''$.  It is clear that if we cancel all $K$-paths in $T''$, we get $T'''$. Words representing the same elements as $U_1,\dotsc ,U_m$ can be obtained from $U'''_1,\dotsc ,U'''_m$  by constructing a joint extension on $v_1, \dotsc, v_t$ by $u_1, \dotsc ,u_t \in K$.

If we identify all $K$-edges in $\Delta$ according to their labels, we get a closed orientable surface $\Delta'$ and an associated graph with circuits, along which we can read  words representing the same elements as $U_1,\dotsc ,U_m$. Let the genus of $\Delta'$ be $n$. By Lemma \ref{genuslemma} we get that 

\begin{equation*}
Genus(u_1, \dotsc, u_t)=n-t-k+1.
\end{equation*}

Hence, a joint extension of genus $Genus(u_1, \dotsc, u_t)+t+1$ is constructed on $v_1, \dotsc, v_t$ by $u_1, \dotsc ,u_t \in K$. The resulting multi-form has genus $n$ (see Definition \ref{multiform}).

Now, we go back to $\Gamma'$. Let $K_1, \dotsc, K_m$ be connected subgraphs of $\Gamma'$ such that each $K_i$ is labeled by elements of some free factor and has no less than two edges. 
We assume that if $K_i$ and $K_j$ are labeled by elements of the same free factor they do not share any edges, nor any vertices. Let $\bar U_1,\cdots ,\bar U_m$ be normal forms of elements $U_1,\dotsc,U_m$. By applying the above process to each $K_i$, it follows that $\bar U_1,\dotsc,\bar U_m$  can be obtained from a multi-form of genus $g$ over the free product by a permissible substitution. This proves the proposition. 

\end{proof}

{\bf Proof of Theorem \ref{th:5}.}

Proposition \ref{th:6} deals with an orientable equation and an orientable set of words. All other cases, where either the equation or the set $(U_1, \dotsc, U_t )$ (or both) are non-orientable, can be proved similarly.  By Proposition \ref{th:6}, it's possible to obtain $V_1,\ldots ,V_m$, in a free product of groups $G$, by a permissible substitution from a multi-form  $\mathcal{A}$ of genus $g$ over $G$. Using this multi-form we are going to construct a solution of the original equation and estimate its length.
 Eventually, after some transformations, we will obtain this solution of the original equation as a solution of a quadratic equation in a free group generated by the union of the generating sets of the free factors.

Let $\Gamma$
be the associated graph. By definition of the multi-form $\mathcal{A}$, it is obtained by separation of all vertices of $\Gamma$ into $p+1$ disjoint sets $B_0,B_1,\dotsc,B_p$, and genus $g_i$ extensions on the vertices of $B_i$, for
$i=1,\dotsc,p$. Let each set $B_i$ have $l_i$ vertices.	

This means that each of the $l_i$ vertices is extended by some word $W(i,s)$, where $s=1,\dotsc,l_i$. To proceed we need the following definition.

\begin{definition}

Let the vertex $v$ be extended by a cyclic word $W$ of some group $H$,
 $\psi_1\dotsb\psi_l=W$ in $H$, $\psi_i \in H$, such that
 $a_1^{-1} \psi_1 a_2,\dotsc,a_{l-1}^{-1} \psi_{l-1} a_l,a_l^{-1} \psi_l a_1$.
We will say that the extended vertex is \emph{augmented} if the words $$a_1^{-1} \psi_1 a_2,\dotsc,a_{l-1}^{-1} \psi_{l-1} a_l,a_l^{-1} \psi_l a_1$$
are replaced by $$A_1^{-1}A_2,\dotsc,A_{l-1}^{-1}A_l,A_l^{-1}WA_1$$
where $A_1=a_1$ and $A_i=\psi_1\psi_2\dotsb\psi_{i-1}a_i$ for $i=2,\dotsc,n-1$.

\end{definition}

\begin{example}
Let $v$ be the vertex extended by a cyclic word
$W$ of a group $H$, $\psi_1\psi_2\psi_3\psi_4=W$ in $H$,$\psi_i \in H$, and
 $a_1^{-1} \psi_1 a_2,a_2^{-1}\psi_2 a_3,a_3^{-1} \psi_3 a_4,a_4^{-1} \psi_4 a_1$.

Augmentation of $v$:

$a_1^{-1} \psi_1 a_2$,

$a_2^{-1}\psi_1^{-1}\psi_1\psi_2 a_3$,

$a_{3}^{-1}\psi_2^{-1}\psi_1^{-1}\psi_1\psi_2 \psi_3 a_4$,

$a_4^{-1} \psi_3^{-1}\psi_2^{-1}\psi_1^{-1}\psi_1\psi_2\psi_3\psi_4 a_1$.

Now we can replace $a_1$ by $A_1$, $\psi_1a_2$ by $A_2$, $\psi_1\psi_2a_3$ by $A_3$,
 $\psi_1\psi_2\psi_3a_4$ by $A_4$ and the product $\psi_1\psi_2\psi_3\psi_4$ by $W$.
\end{example}

Denote the length of the original equation by $M=n(Q)+c(Q)$. If we perform augmentation of every extended vertex of
 $ \mathcal{A}$, then by Proposition 3, the length of every {\it new} letter is bounded from above by $M$, and
the result of augmentation of every extended vertex will look
like the framing set of words, plus the words $W(i,s)$ which
were used for the joint extensions. By Definition 7 (of a joint extension) and Definition 8 (of a multi-form
for a free product), for a fixed $i$, the words $W(i,s)$
satisfy one of the following equalities, for some $\{b_k,c_k,d_j\}$ in a free factor  $G_i$:

\bel{eqn:orientable} \left( \prod_{k=1}^{g_i} [b_k,c_k]
\right)\left(\prod_{j=1}^{l_{i-1}} d_j^\mo W(i,j) d_j\right) W(i,l_i) = 1, \eee   or

\bel{eqn:non-orientable}
 \left(\prod_{k=1}^{g_i} b_k^2\right)
\left(\prod_{j=1}^{l_i-1} d_j^\mo W(i,j) d_j\right) W(i,l_i) = 1. \eee 
Now express $W(i,l_i)$ in terms of commutators, squares and conjugates of $W(i,j), \ j=1,\dotsc, l_{i-1}$, using the equations above, and substitute it into the augumented $ \mathcal{A}$.

Consider the case when all free factors are either free groups,
or abelian groups. Example 4 illustrates the proof. We denote by $M_i$ the length of the $i'th$ extension, $M_i=\sum_{s=1}^{l_i}|W(i,s)|$. The sum  of all $M_i$'s is bounded by $M$.

If $G_i$ is a free group,  then for minimal such $b_k,c_k,d_j$, the length of each of the  $b_k,c_k,d_j$ above is bounded by $2M_i$ in the orientable case, and by  $12 M_i^4$ in the non-orientable case (see \cite{kv}, \cite{lm}).  If $G_i$ is a free abelian group then the first equality becomes  $(\prod_{j=1}^{l_i-1}  W(i,j) )W(i,l_i) = 1,$ and the second equality becomes
$b^2(\prod_{j=1}^{l_i-1}  W(i,j) )W(i,l_i) = 1.$

Let  $V_1^{\prime},\ldots , V_m^{\prime}$ be a set of words
after augmentation and substitution of $W(i,l_i)$ in terms of commutators, squares and conjugates. The set of words  $V_1^{\prime},\ldots , V_m^{\prime}$  is presented by a quadratic set of words in a free group generated by the union of the generating sets of the free factors. Notice that the words $V_1^{\prime},\ldots , V_m^{\prime}$  may be non-reduced.

 The number of  letters after augmentations is bounded by $M$, and substitution of $W(i,l_i)$ in terms of commutators and squares does not give more then $4M$ new letters, so the new total length is not more then $5M$. In a free group, the number of commutators is bounded by the total length of coefficients, and products of commutators in abelian groups are trivial. So the length of each new letter is bounded by $2M$, and then the total length of  $V_1^{\prime},\ldots ,V_m^{\prime}$ is bounded by $10M^2$, in the orientable case. A similar argument shows that in the non-orientable case the number of new letters is bounded by $5M$, but the length of each new letter is bounded by $12M^4$, so therefore
the total length of  $V_1^{\prime},\ldots ,V_m^{\prime}$ is bounded by $60M^5$. For the quadratic equation $Q=1$ in the formulation of the theorem, the tuple $V_1^{\prime},\ldots , V_m^{\prime}$  represents the coefficients. Now we can use the results from \cite{lm}, cited above, for arbitrary quadratic equations in free groups in order to obtain the estimates in Theorem \ref{th:5}. Theorem \ref{th:5} is proved.
 
If the equation $Q=1$ is in standard form, then the estimates are even better: $|\alpha (x)|< N((n(Q)+c(Q))^2$  for an orientable equation, and $|\alpha (x)|< Nn(Q)((n(Q)+c(Q))^5$ for a non-orientable equation.

\begin{example}
Let a solution of a quadratic equation be obtained from the non-orientable multi-form of Example 2 by a permissible substitution:
$$V_1=A\xi_1ED\psi_1C^{-1}B\xi_3A^{-1}FB\xi_2EG_1G_2H_1O_2O_1^{-1}I_1,$$
$$V_2=C\psi_2H_2O_2ZG_2I_2I_1F,$$
$$V_3=D\psi_3H_2H_1^{-1}I_2O_1ZG_1,$$
where $\xi_1\xi_2^{-1}\xi_3=U_1$,$\psi_1\psi_2\psi_3^{-1}=U_2$,
$(U_1,U_2)$ has orientable genus 3 in some free factor $G_i$, and $|V_1|+|V_2|+|V_3|=s$.

Now we  have to bring the multi-form to a quadratic set in a free group.
First of all we perform augmentations, and our multi-form is as
follows:

$$V_1^{\prime}=AE_1DC_1^{-1}B_1U_1A^{-1}FB_1E_1G_1G_2H_1O_2O_1^{-1}I_1,$$
$$V_2^{\prime}=C_1H_3O_2ZG_2I_2I_1F,$$
$$V_3^{\prime}=DU_2^{-1}H_3H_1^{-1}I_2O_1ZG_1,$$
where
$E_1=\xi_1E,B_1=B\xi_2\xi_1^{-1},C_1=C\psi_1^{-1}$, and $H_3=\psi_1\psi_2H_2$. Without loss of generality, we may assume that $U_2=UU_1^{-1}$, where $U=[b_1,c_1][b_2,c_2][b_3,c_3]$. Substituting $U_2$ by $UU_1^{-1}$, we get
a quadratic set of non-orientable genus $13$ for the free group.
The length of every letter in $V_1^{\prime}$, $V_2^{\prime}$, $V_3^{\prime}$
is bounded by $s$, and every letter in $U$ is bounded by $2|U_1|+2|U_2|\leq 2s$.
Using the results for the free group, we get that the length of the
solution is bounded by a polynomial of degree $8$ in $s$.

\end{example}

\section{Toral relatively hyperbolic groups}
We will use the following definition of relative hyperbolicity. A finitely generated group $G$ with  generating set $A$ is relatively hyperbolic relative to a collection of finitely generated subgroups $\mathcal  H=\{H_1,\ldots ,H_k\}$ if the Cayley graph $C(G, A\cup \Pi)$  (where  $\Pi$ is the set of all non-trivial elements of subgroups in $\mathcal H$) is a hyperbolic metric space, and the pair $\{G,\mathcal H\}$ has \emph{Bounded Coset Penetration} property (BCP property for short). 
 The pair $(G,\{H_1,H_2,...,H_k\})$ satisfies the \emph{BCP property}, if   for any $\lambda \geq 1,$
there exists constant $a = a(\lambda)$ such that the following conditions hold. Let $p, q$ be $(\lambda, 0)$-quasi-geodesics without backtracking in $C(G, A\cup \Pi)$  (do not have a subpath that joins a vertex in a left coset of some $H_k$ to a vertex in the same coset (and is not in $H_k$)) such that their initial points coincide ($p_- = q_-$), and for the terminal points $p_+,q_+$ we have  $d_{A}(p_+,q_+) \leq 1.$
 
1) Suppose that for some $ i$, $s$ is a $H_i$-component of $p$ such that $d_A(s_-,s_+) \geq a;$ then there exists a $H_i$-component $t$ of $q$ such that $t$ is connected to $s$ (there exists a path $c$ in $C(G, A\cup \Pi)$ that connects some vertex of $p$ to some vertex of $q$ and the label of this path is a word consisting of letters from $H_i$).
 
2) Suppose that for some $i,$ $s$ and $t$ are connected $H_i$-components of $p$ and $q$ respectively. Then $d_A(s_-,t_-) \leq a$ and $d_A(s_+,t_+) \leq a.$

Recall that a group $G$ that is hyperbolic relative to a collection $\{H_{1},\ldots,H_{k}\}$ of subgroups is called {toral} if $H_{1},\ldots,H_{k}$ are all abelian and $G$ is
torsion-free.

In this section we will prove Theorem \ref{th:2}. Notice that we will use  that the word problem and the conjugacy problem
 in
(toral) relatively hyperbolic groups are decidable.
\begin{prop}\label{Lem:RipsSela2}
Let $\Gamma=\GammaPresentation$ be a
total relatively hyperbolic group and with parabolic subgroups $H_1,\ldots ,H_k$  and $\pi : F(A)\ast H_1\ast\ldots\ast H_k\rightarrow \Gamma$ the canonical epimorphism.  There
is an algorithm that, given a system $S(Z,A)=1$
of equations over $\Gamma$, produces finitely many systems of
equations
\begin{equation}
S_{1} (X_{1},A)=1,\ldots,S_{n}(X_{n},A)=1
\end{equation}
over a free product $P=F\ast H_1\ast\ldots\ast H_k$ and homomorphisms $\rho_{i}: F(Z)\ast P\rightarrow P_{S_{i}}=(F(Z)\ast P)/ncl S_i$ for $i=1,\ldots,n$
such that
\begin{romanenumerate}
\item for every $P$-homomorphism $\phi : P_{S_{i}}\rightarrow P$,  the map $\overline{\rho_{i}\phi\pi}:\Gamma_{S}\rightarrow \Gamma$ is a $\Gamma$-homomorphism, and
\item for every $\Gamma$-homomorphism $\psi: \Gamma_{S}\rightarrow \Gamma$ there is an integer $i$ and an $F$-homomorphism
$\phi : P_{S_{i}}\rightarrow P$ such that $\overline{\rho_{i}\phi\pi}=\psi$.
\end{romanenumerate}
Further, if $S(Z)=1$ is a system without coefficients, the above holds with $G=\GPresentation$ in place of $\Gamma_{S}$ and `homomorphism' in place of
`$\Gamma$-homomorphism'.

Moreover, $|S_i|=O(|S|^2)$ and $|X_i|=O(|S|)$ for each $i=1,\ldots, n.$
\end{prop}

The proof is the same as the proof of Proposition \ref{Lem:RipsSela1}, but instead of
\cite{RS95} one has to use Theorem 3.3 in \cite{Dah} about representatives in a free product of free abelian groups of finite rank. 

Now the proof of Theorem \ref{th:2} is almost identical to the proof of Theorem \ref{th:1}, but instead of the results about length estimates of  a minimal solution of a quadratic system of equations in a free group one should use Theorem \ref{th:5}.

\section{A family of equations over $\Gamma$, for which the Diophantine problem is NP-hard}
The Diophantine problem for a system of equations $S=1$ over a (class of) group(s) $G$, is to determine whether $S=1$ has a solution in $G$. In this section we will complete the proof of Theorem \ref{NPhard} by showing that, for any input of the exact bin packing problem, there is a corresponding quadratic equation $S=1$ over a torsion-free hyperbolic group $\Gamma$, such that a solution to $S=1$ gives a positive answer to the given input, and vice versa. The exact bin packing problem, which is NP-hard (see \cite{KLMT} where the bin packing problem from \cite{Garey}, p. 226, is modified into the exact bin packing problem), is given by:

{\bf Problem} {Exact Bin Packing}
\begin{itemize}
\item INPUT: An s-tuple of positive integers $(r_{1},...r_{s})$ and
positive integers $B$ and $N$.
\item QUESTION: Is there a partition $\{1,\ldots,s\}  =B_{1}\sqcup\cdots\sqcup B_{N}$, such that for each $i=1,\ldots,N$ $$\sum_{j\in B_{i}} r_{j}=B?$$
\end{itemize}

Let $\Gamma=\langle A|R\rangle$ be a non-elementary torsion-free $\delta$-hyperbolic group, and $Cay(\Gamma)$ be the Cayley graph of $\Gamma$ with respect to $A$.  By \cite{Arzh}, $\Gamma$ contains a convex free subgroup $F(b,c)$ of rank two. We can assume that $b$ and $c$ are both cyclically reduced as elements of $\Gamma$, i.e. have minimal length in their respective conjugacy classes, and furthermore that $g^{-1}b^ng\neq c^m$ for all $g\in\Gamma, m,n\in\mathbb{Z}$ (\cite{[BH09]}). Denote fixed minimal words in $A^{\pm1}$ representing these elements by $b$ and $c$ as well. By Theorem I.1.4 of \cite{Ch12} (and the more general Theorem 2.14 of \cite{DGO} given in terms of rotating families), there
exists an integer $D$ such that the normal closure $\langle\langle b^{s_{1} D},c^{s_{2}D}\rangle\rangle$ in $\Gamma$ is free for any $s_{i}>0$.

Given a positive integer $n$, and an $n+2$-tuple $\zeta=(d,\kappa,t_1,\ldots,t_n)$ of positive integers, let $a_{\zeta}=b^{\kappa D}c^{dt_1D}b^{\kappa D}\ldots c^{dt_nD}b^{\kappa D}$. For each bin packing input, we will consider certain systems of equations of a particular form, depending on $\zeta$. Given a bin packing input $(r_1,\ldots,r_s,B,N)$, for each $n$ and $n+2$-tuple $\zeta$, let $S[r_1,\ldots,r_s,B,N,\zeta]=1$ (or just $S[\zeta]=1$) be the equation \begin{equation}\label{*}\prod_{j=1}^sz_j^{-1}[a_{\zeta},b^{dDr_j}]z_j=[a_{\zeta}^{N},b^{dDB}].\end{equation}

Later, we will give explicit conditions on $\zeta$ so that the existence of a solution in $\Gamma$ of $S[\zeta]=1$ implies a solution to the given bin packing input. $|S[\zeta]|$ will be polynomial in $s$, the size of the bin packing input. It is immediate from van Kampen diagrams that a positive solution to the bin packing input $(r_1,\ldots,r_s,B,N)$ implies the existence of a solution to $S[\zeta]=1$ in $\Gamma$, for any $\zeta$.

By Lemma \ref{interLem1}, $S[\zeta]=1$ may be transformed to a system consisting of $s_1$ triangular and $s+1$ constant equations, where $s_1=O(s)$. As in Proposition \ref{Lem:RipsSela1} and Proposition \ref{Bnd}, by then considering canonical representatives $\theta_m,m\leq L=(s_1+s+1)\cdot 2^{5050(\delta+1)^{6}(2|A|)^{2\delta}}$ and all choices of $c_1^{(\ell)},c_2^{(\ell)},c_3^{(\ell)}\in B_{\Gamma}(L)$ (the ball of radius $L$ in $\Gamma$) for which $c_1^{(\ell)}c_2^{(\ell)}c_3^{(\ell)}=1$, a finite number of quadratic systems of equations $S_i[\zeta]=1$, $i=1,\ldots,m_1$ may be constructed. A solution in $\Gamma$ of $S[\zeta]=1$ implies a solution in $F(A)$ of $S_i[\zeta]=1$ for some $i$. Each $S_i[\zeta]=1$ is in variables $X_0=\{x_1^{(1)},x_2^{(1)},x_3^{(1)},\ldots,x_1^{(s_1)},x_2^{(s_1)},x_3^{(s_1)}\}$ with equations of the form:
 \begin{equation}\label{Si}
 \begin{split}
 x_k^{(\ell)}c_k^{(\ell)}(x_{k+1}^{(\ell)})^{-1}&=x_{k'}^{(\ell')}c_{k'}^{(\ell')}(x_{k'+1}^{(\ell')})^{-1}
 \\
(x_{k''}^{(\ell'')})c_{k}^{(\ell'')}(x_{k''+1}^{(\ell'')})^{-1}&=\theta_m([a_{\zeta},b^{dDr_j}])
\\
(x_{k'''}^{(\ell''')})c_{k'''}^{(\ell''')}(x_{k'''+1}^{(\ell''')})^{-1}&=\theta_m([a_{\zeta}^N,b^{dDB}])
\end{split}
\end{equation}
where $k+1$ is taken cyclically with respect to $(1,2,3)$. Note that there are exactly $\frac{1}{2}(3s_1-(s+1))$ many equations of the first type, and $s+1+3s_1$ many coefficients in each $S_i[\zeta]=1$.

If $\zeta$ is changed, only the coefficients of $S[\zeta]$ may differ, so all $S[\zeta]=1$ have the same form of decomposition into triangular equations. Therefore only the coefficients may differ in each $S_i[\zeta]=1$ for $1\leq i \leq m_1$ and any $\zeta$.
\subsection{Entire Transformations}
It's possible to construct a finite number of systems, where every equation in each system includes the canonical representative of exactly one original coefficient from $S[\zeta]=1$, such that the solutions of $S_i[\zeta]=1$ in $F(A)$ factor through the solutions of these systems in an appropriate sense.

\begin{prop}\label{EnTrans}
For each $1\leq i\leq m_1$ and any $\zeta$, if the system $S_i[\zeta]=1$ has a solution in $F(A)$, then there is another quadratic system of equations $\bar E(S_i[\zeta])=1$, in variables $\bar X_i=\{\bar x_1,\ldots,\bar x_{\bar s_i}\}$ , for some $\bar s_i\leq3s_1$, which has a solution $\psi$ in $F(A)$. The equations of  $\bar E(S_i[\zeta])=1$ are given by: 
$$w^{(j)}(\bar X_i)=\theta_m([a_{\zeta},b^{dDr_j}])\text{ for }1\leq j\leq s$$ 
$$w^{(s+1)}(\bar X_i)=\theta_m([a_{\zeta}^N,b^{dDB}])$$
$$w^{(j)}(\bar X_i)=\bar c_{j-s-1}\text{ for }s+2\leq j\leq s+3s_1+1$$

for some $\bar c_1,\ldots,\bar c_{\bar 3s_1}$ in the ball of radius $L$ in $F(A)$. Furthermore, for $\bar X_i^{\psi}=\{\psi(\bar x_1),\ldots,\psi(\bar x_{\bar s_i})\}$, $w^{(j)}(\bar X_i^{\psi})$ is a freely reduced word in $F(A)$ for $j=1,\ldots,s+3s_1+1$.\end{prop}
Note that $\bar E(S_i[\zeta])=1$ has no more variables than $S_i[\zeta]=1$.\begin{proof}This is proven using a rewriting process, described in \cite{Imp}, which is given in terms of generalized equations and combinatorial generalized equations. Since $S_i[\zeta]=1$ is quadratic, the more general process is not needed in full; we describe here the necessary elements for the proof of Proposition \ref{EnTrans}. Note that the notation used here for this simpler version differs slightly from that of \cite{Imp}.
\begin{definition}Given a finite set $A^{\pm1}$, a \emph{combinatorial generalized equation $\tilde\Omega$} consists of the following:
\begin{enumerate}[(i)]
\item A finite set $BS$ of \emph{bases}, which is the disjoint union of \emph{variable bases} $BS_v=\{\mu_1,...,\mu_{2m}\}$ and \emph{constant bases} $BS_c=\{\eta_1,\ldots,\eta_n\}$, for some $m\geq 1,n\geq 0$.

 \item An initial segment $BD=\{1,\dots,\rho_1+1\}$ of $\mathbb{N}$, called the set of \emph{boundaries} of $\tilde\Omega$, where $\rho_1\geq1$. A subset $BD_c=\{\rho_0,\ldots,\rho_1+1\}$ for some $1\leq\rho_0\leq\rho_1+1$ (which may be empty if $n=0$), called the set of \emph{constant boundaries}.

\item A function $\epsilon:BS_v\rightarrow\{-1,1\}$ and an involution $\Delta : BS_v \rightarrow BS_v$. Denote $\Delta(\mu)= \bar\mu$ and call $\mu,\bar\mu$ \emph{dual bases}.

\item Functions $\alpha:BS\rightarrow BD$ and $\beta:BS\rightarrow BD$, where $\alpha(\lambda)<\beta(\lambda)$ for every $\lambda\in BS$, and $\alpha(\eta)\in BD_c$ for every $\eta\in BS_c$ (so $\beta(\eta)\in BD_c$ as well).

\item A map $\sigma:BS_c\rightarrow F(A)$


\end{enumerate}
\end{definition}

 When necessary, denote the set of bases (variable, constant, boundaries, etc.) of $\tilde\Omega$ by $BS(\tilde\Omega)$ (and $BS_v(\tilde\Omega)$, etc.). For any base $\lambda$, boundaries $i$ such that $\alpha(\lambda)<i<\beta(\lambda)$ are called \emph{internal boundaries} of $\lambda$; the \emph{end} boundaries $\alpha(\lambda)$ and $\beta(\lambda)$ can be specified as \emph{initial} and \emph{terminal} boundaries of $\lambda$, respectively. The internal and end boundaries of $\lambda$ are said to be \emph{covered} by $\lambda$. If boundaries $i$ and $i+1$ are covered by $\lambda$, then say the interval $[i,i+1]$ is covered by $\lambda$. If $\alpha(\lambda_1)$ and $\beta(\lambda_1)$ are covered by $\lambda$, say $\lambda_1$ is covered by $\lambda$. Combinatorial generalized equations can be thought of as ``interval diagrams'' (see Figures \ref{fig:genEqn} and \ref{fig:enTrans}).

To each combinatorial generalized equation $\tilde\Omega$, a system of equations\\
$S(x_1,\ldots,x_{\rho_1},A)_{\Omega}=1$ over $F(A)$ is associated, by the following construction:
For each base $\lambda$, let $
   w(\lambda) =(x_{\alpha(\lambda)}\cdots x_{\beta(\lambda)-1})^{\epsilon(\lambda)}$.

\begin{enumerate}[(1)]
\item For each pair of variable bases $\mu,\bar\mu$ form the \emph{basic equation}:
$$w(\mu)=w(\bar\mu)$$

\item For each constant base $\eta$, form the \emph{constant equation}: $$w(\eta)=\sigma(\eta)$$

\end{enumerate}

The variables are referred to as \emph{items}, and $S(x_1,\ldots,x_{\rho_1},A)_{\Omega}=1$ is called the \emph{generalized equation} of $\tilde\Omega$. A generalized equation $S_{\Omega}$ is always assumed to have an associated combinatorial generalized equation $S_{\tilde\Omega}$. Note that a generalized equation is quadratic if and only if each $[i,i+1]$, for $1\leq i\leq \rho_1$, is covered by exactly two bases in $S_{\tilde\Omega}$.

  A \emph{solution} of a generalized equation is a solution $\psi$ of $S_{\Omega}$ in $F(A)$ such that $\psi(w(\lambda))$ is freely reduced in $F(A)$ for each base $\lambda$ (there is no cancellation between each $\psi(x_{i})$ and $\psi(x_{i+1})$ in $\psi(w(\lambda))$). $\psi$ is also considered to be a solution to $S_{\tilde\Omega}$. Following the convention of writing $\psi(x_i)$ as $x_i^{\psi}$, we sometimes denote $\psi(w(\lambda))\in F(A)$ by $\lambda^{\psi}$. Recall that the \emph{length} of a solution $\psi$ is $\sum|\psi|=\sum_{i=1}^{\rho_1}|x_i^{\psi}|$, where $|x_i^{\psi}|$ is the length of the word $x_i^{\psi}$ in $F(A)$. \emph{Minimal} solutions of $S_{\Omega}$ are solutions which are minimal with respect to this length.

\begin{lemma}\label{construct:gen:eqn}
For each $1\leq i\leq m_1$ and any $\zeta$, if $S_i[\zeta]=1$ has a solution $\psi$ in $F(A)$, then there is a generalized equation $S_i[\zeta]_{\Omega}=1$ which has a solution. Furthermore, $\rho_1=3(3s_1-(s+1))+3(s+1)+3s_1$ and the constant bases partition $[\rho_0,\rho_1+1]$ in the corresponding combinatorial generalized equation $S_i[\zeta]_{\tilde\Omega}$. In other words each $[k,k+1]$ is covered by exactly one constant base for $k\geq\rho_0$.
\end{lemma}
\begin{proof}

Given $S_i[\zeta]=1$ with a solution $\psi$ in $F(A)$, consider the $\frac{1}{2}(3s_1-(s+1))$ many equations of the first type described in (\ref{Si}), in some fixed order, followed by the $s+1$ many equations of the second and third type. In each equation, replace the coefficients $c_k^{(\ell)}$ by new variables $y_k^{(\ell)}$ and add the constant equations $y_k^{(\ell)}=c_k^{(\ell)}$ after the equations of the second and third type. 

Now rename every appearance of all variables in this system, in the order of equations, with variables in each equation ordered left to right, introducing additional equations $x_j=x_{j'}$ for each variable appearing twice (i.e. $x_k^{(\ell)}$ is renamed as $x_j$ and $x_{j'}$ in two different equations). So now a generalized equation is obtained, call it $S_i[\zeta]_{\Omega'}$ where $F_{R(S_i[\zeta])}=F_{R(S_i[\zeta]_{\Omega'})}$. $S_i[\zeta]_{\Omega'}$ corresponds to a combinatorial generalized equation $S_i[\zeta]_{\tilde\Omega'}$ pictured in Figure \ref{fig:genEqn}. Note that $\rho_0=3(3s_1-(s+1))-1$, since there are $\frac{1}{2}(3s_1-(s+1))$ many equations without constants, each with $6$ variables. Then there are $s+1$ constant bases, labeled by canonical representatives of commutators, each covering $3$ variable bases. Finally there are $3s_1$ constant bases, labeled by the $c_k^{(\ell)}$, each covering one variable base. So $\rho_1=3(3s_1-(s+1))+3(s+1)+3s_1$. The constant bases partition the interval $[\rho_0,\rho_1+1]$. There are $\frac{1}{2}(3s_1-(s+1))+3s_1+3s_1=\frac{15}{2}s_1-\frac{s}{2}-\frac{1}{2}$ many pairs of dual variable bases.
  
In Figure \ref{fig:genEqn}, the dual bases $\mu_1$ and $\bar\mu_1$ give the basic equation $x_1x_2x_3=x_4x_5x_6$, corresponding to the first equation $x_k^{(\ell)}c_k^{(\ell)}(x_{k+1}^{(1)})^{-1}=x_{k'}^{(\ell')}y_{k'}^{(\ell')}(x_{k'+1}^{(\ell')})^{-1}$ of $S_i[\zeta]=1$. The appearance of $x_k^{(\ell)}$ in another equation (in this example another equation of the first type) is represented by the dual bases $\mu_2$ and $\bar\mu_2$. The constant base $\eta_{s+2}$ is labeled by $\sigma(\eta_{s+2})=c_k^{(\ell)}$ and the dual bases $\mu_3$ and $\bar\mu_3$ give the constant equation $x_2=c_k^{(\ell)}$. Finally, the constant base $\eta_1$ is labeled by $\sigma(\eta_1)=\theta_m([a_{\zeta},b^{dDr_1}])$, and the dual bases $\mu_4$ and $\bar\mu_4$ represent appearances of some variable $x_{k''}^{(\ell'')}$ in an equation of the first type and in an equation of the second type.

Now the generalized equation $S_i[\zeta]_{\Omega'}$ might not have a solution. This occurs if, for $x_j^{\psi}x_{j+1}^{\psi}x_{j+2}^{\psi}$ on one side of some equation, there is cancellation between $x_j^{\psi}$ or $x_{j+2}^{\psi}$, and $x_{j+1}^{\psi}=c_k^{(\ell)}$. However, there is a generalized equation $S_i[\zeta]_{\Omega}$ of exactly the same form, except that $\sigma(\eta_{s+2})=\bar c_1^{(1)},\ldots, \sigma(\eta_{s+1+3s_1})=\bar c_3^{(s_1)}$ for some subwords $\bar c_k^{(\ell)}$ of $c_k^{(\ell)}$, which does a solution.

\end{proof}

\begin{figure}[ht!]
\labellist
\small\hair 2pt
\pinlabel $1$ at 5 60
\pinlabel $2$ at 19 60
\pinlabel $3$ at 33 60
\pinlabel $\ldots$ at 47 60

\pinlabel $\rho_0$ at 189 60

\pinlabel $\rho_1$ at 342 60

\pinlabel $\mu_1$ at 26 47
\pinlabel $\bar\mu_1$ at 68 47

\pinlabel $\mu_2$ at 12 25
\pinlabel $\bar \mu_2$ at 154 25

\pinlabel $\eta_1$ at 210 47
\pinlabel $\mu_3$ at 26 25
\pinlabel $\bar\mu_3$ at 308 25
\pinlabel $\mu_4$ at 125 25
\pinlabel $\bar\mu_4$ at 225 25
\pinlabel $\eta_{s+2}$ at 308 47

\pinlabel $\eta_{s+1+3s_1}$ at 349 47

\endlabellist
\centering
\includegraphics[scale=.96]{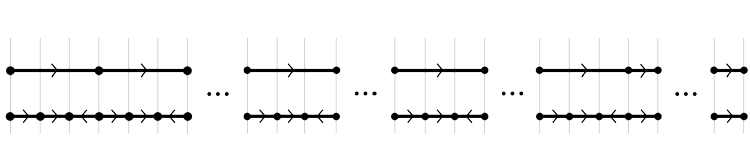}
\caption{}
\label{fig:genEqn}
\end{figure}

We now describe the rewriting process which may be applied to each $S_i[\zeta]_{\Omega}$ with a solution. Notice that in each $S_i[\zeta]_{\tilde\Omega}$, every boundary is an end boundary of some base.

\begin{definition}Rewriting process (Entire transformation): For a quadratic generalized equation $S_{\Omega}$ with a solution $\psi$, the process is described by transformations of $S_{\tilde\Omega}$ and construction of factoring homomorphisms to resulting coordinate groups 

Starting at $1\in BD(\tilde\Omega)$, if $1=\rho_0$, terminate the process. Otherwise, there are variable bases $\mu_1$ and $\mu_2$ with $\alpha(\mu_1)=\alpha(\mu_2)=1$ and $\beta(\mu_1)\geq\beta(\mu_2)=2$. There are three possible cases for $\mu_1$ and $\mu_2$.
\begin{enumerate}[(i)]
\item If $\beta(\mu_1)=2$ and $\mu_2=\bar\mu_1$, remove the pair $\mu_1$ and $\bar\mu_1$ (they are ``matched bases'' and $\epsilon(\mu_1)=\epsilon(\bar\mu_1)$). Then move all bases to the left by one, and decrease $\rho_0,\rho_1$ by one, to obtain a generalized equation $S_{\Omega'}$. Define a homomorphism $\pi':F_{R(S_{\Omega})}\rightarrow F_{R(S_{\Omega'})}*\langle t\rangle$ by $\pi'(x_i)=x'_{i-1}$ for $1<i$, and $\pi'(x_1)=t$, where $t$ is a new free variable. 

\item If $\beta(\mu_1)=2$ and $\mu_2\neq\bar\mu_1$, replace $\bar\mu_1$ by $\mu_2$, reversing $\epsilon(\mu_2)$ if $\epsilon(\mu_1)\neq\epsilon(\bar\mu_1)$, before removing the pair $\mu_1$ and $\bar\mu_1$. Then move all bases to the left by one, and decrease $\rho_0$ and $\rho_1$ by one to obtain a generalized equation $S_{\Omega'}$. Define a homomorphism $\pi':F_{R(S_{\Omega})}\rightarrow F_{R(S_{\Omega'})}$ by $\pi'(x_i)=x'_{i-1}$ for $1<i$, and $\pi'(x_1)=w'(\bar\mu_1)$, where if $w(\bar\mu_1)=(x_k\cdots x_{k+r})^{\epsilon(\bar\mu_1)}$, then $w'(\bar\mu_1)=(x'_{k-1}\cdots x'_{k+r-1})^{\epsilon(\bar\mu_1)}$.

\item Otherwise $2<\beta(\mu_1)$. Let $a=\alpha(\bar\mu_1)$ and $b=\beta(\bar\mu_1)$. \emph{Cut} $\mu_1$ at $2$ into bases $\mu'_1$ and $\mu''_1$, considering each possibility for where to cut $\bar\mu_1$ at $j$ into $\bar\mu'_1$ and $\bar\mu''_1$. See Figure \ref{fig:enTrans}, ($\epsilon$ of each base is not pictured, though cutting and transferring bases must agree with orientation). Consider cuts between existing boundaries $k$ and $k+1$ by letting $j=k+1$ and increasing by one all $\alpha(\lambda),\beta(\lambda)\geq k+1$, as well as $\rho_0$ and $\rho_1$. Then \emph{transfer} $\mu_2$ onto $\bar\mu_1$, i.e. let $\alpha(\mu_2)=a$ and $\beta(\mu_2)=j$. Delete the lone base $\mu'_1$, as well as $\bar\mu'_1$ and rename $\mu''_1$ and $\bar\mu''_1$ as $\mu_1$ and $\bar\mu_1$. Now move all bases to the left by one, and decrease $\rho_0,\rho_1$ by one

For each $S_{\tilde\Omega'}$ constructed in this case, let $\pi':F_{R(S_{\Omega})}\rightarrow F_{R(S_{\Omega'})}$ be defined by $\pi'(x_1)=w'(\bar\mu'_1)$ and $\pi'(x_i)=x'_{i-1}$ for $i>1$ if $j$ is an existing boundary. If $j$ is inserted between $k$ and $k+1$, let $\pi'(x_i)=x'_{i-1}$ for $1<i\leq k$, and $\pi'(x_i)=x'_{i}$ for $k+1\leq i\leq \rho_1$.

\end{enumerate}

The resulting generalized equations are all quadratic, so apply the same process to each.

\end{definition}

\begin{figure}[ht!]
\labellist
\scriptsize\hair 2pt

\pinlabel $1$ at 135 88

\pinlabel $a$ at 7 41
\pinlabel $j$ at 14 41

\pinlabel $a$ at 71 41
\pinlabel $j$ at 85 41

\pinlabel $a$ at 135 41
\pinlabel $j$ at 156 41
\pinlabel $a$ at 200 41
\pinlabel $j$ at 227 41
\pinlabel $a$ at 264 41
\pinlabel $j$ at 297 41
\pinlabel $\mu'_1$ at 143 80
\pinlabel $\mu''_1$ at 163 80

\pinlabel $\bar\mu'_1$ at 143 80
\pinlabel $\bar\mu''_1$ at 163 80

\pinlabel $\mu_2$ at 143 67

\endlabellist
\centering
\includegraphics[scale=1.12]{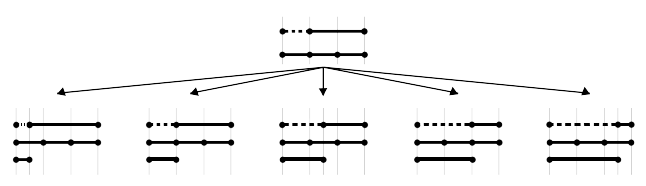}
\caption{}
\label{fig:enTrans}
\end{figure}

For each $S_i[\zeta]_{\Omega}$, the iterated process can be represented as a rooted finite valence tree $\mathcal{T}_i$ with generalized equations at vertices and edges labeled by surjective homomorphisms of coordinate groups. Any solution of $S_i[\zeta]_{\Omega}$ factors through some branch of $\mathcal{T}_i$. If the process terminates for some branch $b$, call the terminal generalized equation $E_b(S_i[\zeta]_{\Omega})$. If a solution of $S_i[\zeta]_{\Omega}$ factors through a branch $b$, the solutions to each consecutive generalized equation have strictly decreasing length. Furthermore, if the rewriting process terminates with $E_b(S_i[\zeta]_{\tilde\Omega})$, notice that $[1,\rho_1+1]$ is partitioned by constant bases, and each constant base is partitioned by variable bases, with each variable base $\mu$ covering exactly one item (i.e. $\beta(\mu)=\alpha(\mu)+1$ and $\mu^{\psi}=x^{\psi}$ for a variable base $\mu$ covering $x$ and a solution $\psi$).

\begin{lemma} \cite{Imp}\label{entire:trans} If the length of $\psi$ is minimal with respect to $Aut(F_{R(S_i[\zeta]_{\Omega})})$, the rewriting process terminates for any branch $b$ through which $\psi$ factors.\end{lemma}
\begin{proof}Suppose $\psi$ factors through an branch of $\mathcal{T}_i$ where the rewriting process continues infinitely. Since there are finitely many boundaries introduced by transferring and cutting bases, the number of bases may only decrease and the number of boundaries is bounded from above. So there is a finite number of possible distinct generalized equations and if $\psi$ factors through an infinite branch, the same generalized equation must be repeated. However, each step of the rewriting process shortens length of solutions, contradicting our assumption that $\psi$ is minimal with respect to $Aut(F_{R(S_i[\zeta]_{\Omega})})$.
\end{proof}

\begin{remark}Notice that the entire transformation does not add more bases, and may only remove a pair of occurences of some variable (and its inverse) by deleting a pair of dual bases, so the resulting system of equations is quadratic.
\end{remark}
So for each $S_i[\zeta]=1$ with a solution in $F(A)$, there is a generalized equation $S_i[\zeta]_{\Omega}$ which has a solution $\psi$, and $\psi$ factors along some terminating branch $b$ of $\mathcal{T}_i$. So $\bar E(S_i[\zeta])=E_b(S_i[\zeta]_{\Omega})$ satisfies the conclusion of Proposition \ref{EnTrans} after renaming variables using the basic equations, since each variable base covers exactly one item. Also let $\bar E(S_i[\zeta])_{\tilde\Omega}=E_b(S_i[\zeta]_{\tilde\Omega})$.
\end{proof}

\subsection{Disk diagrams}

In \cite{Olshanskii-1989} (see also Theorem 2.1 in \cite{KLMT}), it's shown that an orientable quadratic equation $S=1$ in standard form:  \begin{equation}\label{**}\prod_{i=1}^g[x_i,y_i]\prod_{j=1}^{m-1}z_j^{-1}C_jz_jC_m=1\end{equation}
has a solution in $F(A)$ if and only if there is a collection of disks with oriented boundaries labeled by freely reduced words $C_1,\ldots,C_m$, which tile a union of surfaces $\Sigma_1,\ldots,\Sigma_l$, where gluings between boundaries of disks must respect labelings. Furthermore, for the \emph{reduced Euler characteristic} of $S$ defined by $\bar\chi(S)=2-2g$, the Euler characteristics of the surfaces satisfy the inequality $\sum_{i=0}^l\chi(\Sigma_i)\geq 2-2g+2l$. Since any orientable quadratic equation $S=1$ may be sent to a standard quadratic equation $\bar S=1$ by an automorphism of $F_{R(S)}$, let $\bar\chi(S)=\bar\chi(\bar S)$. For a solution $\psi$ of a system $S=1$, denote such a corresponding \emph{disk diagram} by $D(S,\psi)$.

\begin{lemma}\label{reps:diagram}
For each $1\leq i \leq m_1$ and any $\zeta$, if $\psi$ is a solution of $\bar E(S_i[\zeta])$, there is a disk diagram $D(\bar E(S_i[\zeta]),\psi)$ tiling a union of spheres.
\end{lemma}
\begin{proof}
Each coefficient in $\bar E(S_i[\zeta])$ is graphically equal to a word $w^{(j)}(\bar X_i^{\psi})$. So if the boundaries of $s+3s_1+1$ disks are labeled by the coefficients of $\bar E(S_i[\zeta])$, the boundary of each disk may be partitioned by solutions of variables. Since each variable appears exactly twice, these disks may be glued respecting labelings, and so tile some union of surfaces $\Sigma=\Sigma_1\cup\cdots\cup\Sigma_l$.

After transforming $S[\zeta]=1$ to triangular equations $z_1^{(k)}z_2^{(k)}z_3^{(k)}=1$ over $\Gamma$, a solution to $S[\zeta]=1$ in $\Gamma$ implies a solution over $F(A)$ to equations $z_1^{(k)}z_2^{(k)}z_3^{(k)}=c_1^{(k)}c_2^{(k)}c_3^{(k)}$  formed by canonical representatives. So the coordinate group of the equation $P=1$ defined by: $$\prod_{j=1}^s z_j^{-1}\theta_m([a_{\zeta},b^{dDr_j}])z_j(\theta_m([a_{\zeta}^N,b^{dDB}]))^{-1}= c_1^{(1)}\cdots c_3^{(s_1)}$$ is a  subgroup of $F_{R(S_i[\zeta]_{\Omega})}$. Its coordinate group over $F$ is $F_{R(P)}$ and $\bar\chi(P)=2$.  The group $F_{R(S_i[\zeta]_{\Omega})}$ is a free product of the coordinate group of a standard quadratic equation, with a free group. Moreover, since $F_{R(P)}$ is freely indecomposable, it must be conjugate into the coordinate group of this standard quadratic equation. Therefore $\bar\chi(S_i[\zeta]_{\Omega})=2$. The homomorphisms constructed in the rewriting process are all either isomorphisms or map $F_{R(S_i[\zeta]_{\Omega})}$ to a free product of the coordinate group of the standard quadratic equation, with a free group of smaller rank. So these homomorphisms preserve reduced Euler characteristic and $\bar\chi(\bar E(S_i[\zeta]))=2$, implying $\Sigma_1,\ldots,\Sigma_l$ must all be spheres.

\end{proof}
 
For a finitely generated group $G=\langle A\rangle$, and a (not necessarily reduced) word $w=w(A^{\pm1})$ labeling a path $p$ in $Cay(G,A)$, let $\|w\|=\|p\|$ denote the word length of $w$, or equivalently the path length of $p$. Let $w_1\equiv w_2$ denote graphical equality of words in $G$. Denote the geodesic word length of $w$, or equivalently the path length of a geodesic with the same endpoints as $p$, by $|w|$. Abusing notation slightly, if $p$ is a geodesic, that fact may be emphasized by the notation $|p|$ for path length. The notation $\|w\|_G$ and $|w|_G$ is used when it's necessary to distinguish the group. Order the vertices of paths in $Cay(G,A)$ by letting $q_1<q_2$, for $q_1$ and $q_2$ on $p$, if and only if $q_1$ is between $q_2 $ and the initial vertex of $p$, or if $q_1$ is the initial vertex. Denote the subpath of $p$ from $q_1$ to $q_2$ by $p_{(q_1,q_2)}$.

\begin{prop}\label{coms:diagram}
There is a constant $K_0$, which only depends on $\Gamma$ and the choice of $b,c,$ and $D$, such that for $d\geq \kappa>K_0$ and any $(t_1,\ldots,t_n)$, if $S[\zeta]=1$ has a solution in $\Gamma$ for the corresponding $\zeta$, there is a generalized equation $\hat E(S_i[\zeta])$ with a solution $\hat\psi$. The coefficients of $\hat E(S_i[\zeta])$ are: $$[a_{\zeta},b^{dDr_1}],\ldots,[a_{\zeta},b^{dDr_s}], [a_{\zeta}^N,b^{dDB}]$$along with some $R_1,\ldots,R_{K_2}\in F(A)$, with $|R_1|,\ldots,|R_{K_2}|\leq\mathcal{L}$ for some constants $K_2,\mathcal{L}=O(s)$. Also, $\hat E(S_i[\zeta])_{\tilde\Omega}$ has $K_1$ variable bases for some constant $K_1=O(s)$, and gluing disks with boundaries labeled by the coefficients of $\hat E(S_i[\zeta])$ along solutions of dual variable bases yields a disk diagram $\hat D=D(\hat E(S_i[\zeta]),\hat\psi)$ tiling a union of spheres.
\end{prop}

\begin{proof}The solution of $S[\zeta]=1$ in $\Gamma$ factors through a solution of some $S_i[\zeta]=1$, $i\leq m_1$, and so by Proposition \ref{EnTrans}, there is a generalized equation $\bar E(S_i[\zeta])$ with a solution $\psi$. By Lemma \ref{reps:diagram}, there is a disk diagram $\bar D= D(\bar E(S_i[\zeta]),\psi)$ tiling a union of spheres $\Sigma=\Sigma_1\cup\cdots\cup\Sigma_l$. 

Recall that every variable base of $\bar E(S_i[\zeta])$ covers just a single item. For $1\leq j\leq s+3s_1+1$, let $\bar w^{(j)}=w^{(j)}(\bar X_i^{\psi})$ and let $p_j$ be paths in $Cay(\Gamma)$ labeled by $\bar w^{(j)}$ where double occurrences of a variable correspond to a shared subpath. In other words, suppose $w^{(j)}(\bar X_i)=x_{j_1}\cdots x_{j_{n(j)}}$ and $w^{(j')}(\bar X_i)=x_{j'_1}\cdots x_{j_{n(j')}}$ for integers $n(j),n(j')$ and $x_{j_1},\ldots,x_{j_{n(j)}},x_{j'_1},\ldots,x_{j_{n(j')}}\in\bar X_i$. If $x_{j_k}=(x_{j'_{k'}})^{\pm 1}$ for some $1\leq k\leq n(j),1\leq k'\leq n(j')$, then $p_j$ and $p_{j'}$ share the subpath labeled by $x_{j_k}^{\psi}$ (which may have opposite orientations on the paths if $x_{j_k}=(x_{j'_{k'}})^{-1}$). On each $p_j$, let $q_{j_k}$ be the terminal vertex of $x_{j_k}^{\psi}$ for $1\leq k\leq n(j)$, and $q_{j_0}$ be the initial vertex of $p_j$.

By \cite{Olsh} and \cite{RS95}, there are constants
$\pi$ and $\epsilon$ such that any path in $Cay(\Gamma)$ labeled by powers of $b$ and $c$, or by $\theta_m([a_{\zeta},b^{dDr_j}])$, $\theta_m([a_{\zeta}^N,b^{dDB}])$ for $m\leq L$ and $1\leq j \leq s$, are $(\pi,\epsilon)$-quasigeodesic. There is a constant $M=M(\delta,\pi,\epsilon)$, where $\delta$ is the constant of hyperbolicity for $\Gamma$, such that any two $\pi,\epsilon$-quasigeodesics with the same endpoints are in the $M$-neighborhood of each other with respect to the Hausdorff metric (\cite{[BH09]}).

Let $\mathcal{C}_j=[a_{\zeta},b^{dDr_j}]$ for $1\leq j\leq s$, and $\mathcal{C}_{s+1}=[a_{\zeta}^N,b^{dDB}]$. There are paths $\hat p_j$, labeled by $\mathcal{C}_j$, with the same endpoints as $p_j$, for $1\leq j\leq s+1$. So for $1\leq j\leq s+1$ and $0\leq k\leq n(j)$, there are minimal geodesics $s_{j_k}$ from $q_{j_k}$ on $p_j$, to vertices $\hat q_{j_k}$ on $\hat p_j$, with $\|s_{j_k}\|\leq M$. Note that $\hat q_{j_0}=q_{j_0}$ and $\hat q_{j_{n(j)}}=q_{j_{n(j)}}$.

\begin{lemma}\label{no:twist}
Let $p$ and $p'$ be $(\pi,\epsilon)$-quasigeodesics in $Cay(\Gamma)$ with the same initial and terminal vertices. Suppose $q_1<q_2$ are vertices on $p$ and there are paths $\gamma_1$ from $q_1$ to $q'_1$, and $\gamma_2$ from $q_2$ to $q'_2$, where $q'_2<q'_1$ are vertices on $p'$ and $\|\gamma_1\|,\|\gamma_2\|\leq M$. Then $\|p_{(q_1,q_2)}\|\leq\mathcal{N}=2M(\pi+\pi^2)+\epsilon(\pi+1)+\pi$.
\end{lemma}

\begin{proof}
Let $q_0<q_1$ be the maximal (with respect to ordering of vertices on $p$) vertex on $p$ with a path of length at most $M$ from $q_0$ to a vertex $q'_0$ on $p'$ such that $q'_0<q'_2$. Now $\|p_{(q_0,q_1)}\|+\|p_{(q_1,q_2)}\|\leq\pi(\|p'_{(q'_0,q'_2)}\|+2M)+\epsilon$. But $\|p'_{(q'_0,q'_2)}\|\leq\pi(2M+1)+\epsilon$, since for any $q_0<q_3$ on $p$, there is a vertex $q'_2\leq q'_3$ on $p'$ with a path of length at most $M$ from $q_3$ to $q'_3$. The inequality given in the lemma follow from these inequalities.
\end{proof}
In general, $k<k'$ need not imply that $\hat q_{j_k}<\hat q_{j_{k'}}$, but Lemma \ref{no:twist} shows that there are paths of bounded length which partition $p_j$ and $\hat p_j$ in the same order (though since $\hat q_{j_k}$ may equal $\hat q_{j_{k'}}$ for $k\neq k'$, some of the partitioned subpaths of $\hat p_j$ may be empty).
\begin{cor}\label{same:partition}
For each $1\leq j\leq s+1$, the vertices $\hat q_{j_0},\ldots,\hat q_{j_{n(j)}}$ may be renamed so that if $k<k'$, then $\hat q_{j_k}\leq\hat q_{j_{k'}}$, and after renaming there are geodesics $\hat s_{j_k}$ from $q_{j_k}$ to $\hat q_{j_k}$, with $\|\hat s_{j_k}\|\leq R=\mathcal{N}+M$.
\end{cor}

The following fact about free groups is needed.

\begin{lemma}\label{prod:free}Let $w_1=w_2$ in a rank $k$ free group $F(x_1,\ldots,x_k)$, where $w_1=x_{\alpha}^{r_1}x_{\beta}^{\overline{r}_1}\cdots x_{\alpha}^{r_n}x_{\beta}^{\overline{r}_n}$ , $\alpha\neq\beta$, and $w_2=u_1^{s_1}u_2^{\overline{s}_1}\cdots u_1^{s_m}u_2^{\overline{s}_m}$ for freely reduced words $u_1(x_1,\ldots,x_k)$ and $u_2(x_1,\ldots,x_k)$ with $m,n\geq 1$. Only $\bar r_n$ and $\bar s_m$ may be zero in the case where $m,n>1$, and $r_i+\bar r_i>\|u_1\|,\|u_2\|$ for $1\leq i\leq n$ and $\bar r_i+r_{i+1}>\|u_1\|,\|u_2\|$ for $1\leq i\leq n-1$ (for non-zero $r_i,\bar r_i$). Let $R$ be the Hausdorff distance between a path in $Cay(F)$ labeled by $w_1$ and a path with the same endpoints labeled by $w_2$. If $\|u_1\|(s_i-3)\geq2R$ for some $i$, then $u_1=v_1^{-1}x_j^tv_1$ and if $\|u_2\|(\bar s_{i'}-3)\geq2R$ for some $i'$, then $u_2=v_2^{-1}x_{j'}^tv_2$, for $j,j'=\alpha\text{ or }\beta$, $v_1,v_2\in F(x_1,\ldots,x_k)$, and some non-zero integer $t$.
\end{lemma}
\begin{proof}
Since $F$ is free, $u_1,u_2$ must each contain either $x_{\alpha},x_{\beta}$. Since there may be cancellation of at most $R$ between sequential powers of $u_1$ and $u_2$, $w_1$ contains subwords $u_1^2$ or $u_2^2$, corresponding to the $i$ or $i'$ for which $\|u_1\|(s_i-3)>2R$ or $\|u_2\|(\bar s_{i'}-3)>2R$. So if $u_1$ ($u_2$) is only trivially conjugate and contains both $x_{\alpha}$ and $x_{\beta}$, (it may not contain any other generator since $w_1\in\langle x_{\alpha},x_{\beta}\rangle$), then there is some  $x_{\alpha}^{r_j}x_{\beta}^{\bar r_j}$ or $x_{\beta}^{\bar r_j}x_{\alpha}^{r_{j+1}}$ which is a subword of $u_1$ ($u_2$), but $u_1$ and $u_2$ are too short.
\end{proof}

Now for each $x\in\bar X_i$, if $\|x^{\psi}\|>L$, then $x^{\pm 1}$ appears as $x_{j_k}$ and $x_{j'_{k'}}$ in $w^{(j)}(\bar X_i)$ and $w^{(j')}(\bar X_i)$ for $1\leq j,j'\leq s+1$. Let $p_x$ be the subpath of $\hat p_j$ from $\hat q_{j_{k-1}}$ to $\hat q_{j_k}$. If $x_{j_k}=x_{j'_{k'}}$, then let $p'_x$ be the subpath of $\hat p_{j'}$ from $\hat q_{j'_{k'-1}}$ to $\hat q_{j_{k'}}$, whereas if $x_{j_k}=(x_{j'_{k'}})^{-1}$, let $p'_x$ be the subpath of $\hat p_{j'}$ oriented in reverse, from $\hat q_{j_{k'}}$ to $\hat q_{j'_{k'-1}}$. There are geodesics $t_1$ from the initial vertex of $p'_x$ to $\hat q_{j_{k-1}}$, and $t_2$ from the terminal vertex of $p'_x$ to $\hat q_{j_k}$, with $\|t_1\|,\|t_2\|\leq2R$. Notice that $t_1^{-1}p'_xt_2$ is a $(\pi,\epsilon+4R(\pi+1))$-quasigeodesic with the same endpoints as $p_x$ (which is $(\pi,\epsilon)$-quasigeodesic). So there is a constant $M_1=M_1(\delta,\pi,\epsilon+4R(\pi+1))$ such that $p_x$ and $t_1^{-1}p'_xt_2$ are in the $M_1$-neighborhood of each other. 

Since $\hat p_j,\hat p_{j'}$ are labeled by $\mathcal{C}_j$ and $\mathcal{C}_j'$ for some $1\leq j,j'\leq s+1$, $\hat p_j,\hat p_{j'}$ may each be partitioned into subpaths where each subpath is labeled by either $b^D$ or $c^D$. Call the vertices separating these subpaths \emph{$D$-vertices}. Let $D_1=D\max\{|b|,|c|\}$. Now for each $D$-vertex $v$ on $p_x$ there is a geodesic $s_v$ to $v$ from a vertex $u$ on $t_1p'_xt_2^{-1}$ with $\|s_v\|\leq M_1$. For each $D$-vertex $v$ on $p_x$ further than $\pi(2R+M_1)+\epsilon$ (measured along $p_x$) from both the initial and terminal vertices of $p_x$, $u$ is actually on $p'_x$ and $u$ is separated from a $D$-vertex $v'$ of $p'_x$ by a subpath of length less than $D_1$. So for each $D$-vertex $v$ on $p_x$, other than those within $\pi(2R+M_1)+\epsilon$ of end vertices of $p_x$, there is a geodesic $t_v$ to $v$ from a $D$-vertex $v'$ on $p'_x$ with $|t_v|<C=M_1+D_1$. Call each geodesic to a $D$-vertex of $p_x$ from a $D$-vertex of $p'_x$ a \emph{$D$-geodesic}. Let $B_C=|B_{\Gamma}(C)|$ be the size of the ball of radius $C$ in $\Gamma$.

\begin{lemma}\label{bnd:error}
Suppose $S[\zeta]=1$ has a solution in $\Gamma$ for $\zeta=(d,\kappa,t_1,\ldots,t_n)$ where $$\kappa>K_0=\max\{\pi(D_1+2C)+\epsilon,\frac{\pi(2D_1(2M+3)+2C)+\epsilon}{2D_1},\frac{\mathcal{N}}{2D_1}\}$$ $d\geq\kappa$ and $n,t_1,\ldots,t_n$ are any positive integers. Let $\psi$ be a solution of $\bar E(S_i[\zeta])$ for some $1\leq i\leq m_1$. Suppose $\|x^{\psi}\|>\max\{L_1,L\}$ for some $x\in\bar X_i$, where $$L_1=\frac{1}{\pi}(2(\pi(2R+M_1)+\epsilon)+D_1(B_C(2\kappa+1)+2)-\epsilon)-2R.$$ Then the corresponding subpaths $p_x$ and $p'_x$ are graphically equal up to bounded error. In other words, $p_x$ is labeled by $$c_{j_0}\bar w_{x_1}c_{j_1}\cdots c_{j_{\mathcal{R}(x)-1}}\bar w_{x_{\mathcal{R}(x)}}c_{j_{\mathcal{R}(x)}}$$ and $p'_x$ is labeled by $$c'_{j_0}\bar w_{x_1}c'_{j_1}\cdots c'_{j_{\mathcal{R}(x)-1}}\bar w_{x_{\mathcal{R}(x)}}c'_{j_{\mathcal{R}(x)}}$$ for some $\mathcal{R}(x)\leq B_C$, where $\sum_{i=0}^{\mathcal{R}(x)}\|c_{j_i}\|<\mathcal{E}_1$ and $\sum_{i=0}^{\mathcal{R}(x)}\|c'_{j_i}\|<\mathcal{E}$ for\\ $\mathcal{E}_1=B_C(D_1(2\kappa+3)+2C)+2(\pi(2R+M_1)+\epsilon)+D_1$ and\\ $\mathcal{E}=\pi(B_C(4C+2\kappa D_1)+2\pi(2R+M_1)+D_1+4R)+(2B_C+2\pi+1)\epsilon+2B_C(C+D_1)$.
\end{lemma}
\begin{proof}Suppose $\|x^{\psi}\|>\max\{L_1,L\}$ and $p_x,p'_x$ are the subpaths (with orientation possibly reversed) of $\hat p_j$ and $\hat p_{j'}$ (which are labeled by $\mathcal{C}_j$ and $\mathcal{C}_{j'}$, respectively) as defined previously. There are at least $2\kappa(B_C+1)$ $D$-geodesics to distinct $D$-vertices of $p_x$ from $D$-vertices of $p'_x$ (which may not be distinct), each of length less than or equal to $C$. So there must be multiple $D$-geodesics with the same label in $\Gamma$, to $D$-vertices on $p_x$ which are separated by a subpath of length at least $2\kappa D_1$.

Suppose $h\in B_{\Gamma}(C),h\neq 1$, labels a pair of $D$-geodesics $t_{v_1}$ and $t_{v_2}$ to $D$-vertices $v_1$ and $v_2$ which are at least $2\kappa D_1$ apart on $p_x$. Assume $v_1$ and $v_2$ are the furthest apart pair of $D$-vertices whose $D$-geodesics are labeled by $h$, and $v'_1$ and $v'_2$ are their corresponding $D$-vertices on $p'_x$. Let $w_1(b^D,c^D)$ be the subword of $\mathcal{C}_j$ labeling the subpath of $p_x$ from $v_1$ to $v_2$, and let $w_2(b^D,c^D)$ be the subword of $(\mathcal{C}_{j'})^{\pm1}$ labeling the subpath of $p'_x$ from $v'_1$ to $v'_2$. Note that $h^{-1}w_2(b^D,c^D)h=w_1(b^D,c^D)$ in $\Gamma$, since $\kappa$ is large enough to not permit ``twisting'' (see Lemma \ref{no:twist}), and that $\|w_1\|\geq2\kappa D_1$, $\|w_2\|\geq2D_1(2M+3)$.

Since the normal closure $\langle\langle b^D,c^D\rangle\rangle$ in $\Gamma$ is free, the subgroup \\$H=\langle b^D,c^D, h^{-1}b^Dh,h^{-1}c^Dh\rangle\leq\langle\langle b^D,c^D \rangle\rangle$ is free. If $\{b^D,c^D, h^{-1}b^Dh,h^{-1}c^Dh\}$ is a basis of $H$, then $h^{-1}w_2h=w_1$ in $F(A)$.
Otherwise, without loss of generality, either $\{b^D,c^D\}$ is a basis of $H$ or $\{b^D,c^D, h^{-1}b^Dh\}$ is a basis of $H$. We will show that in both cases $h^{-1}w_2h=w_1$ in $F(A)$ (the proof is the same for a basis $\{b^D,c^D, h^{-1}c^Dh\}$).

By \cite{Olsh}, if $w_1$ or $w_2$ is in either $\langle b^D\rangle$ or $\langle c^D\rangle$, then $h^{-1}w_2h=w_1$ in $F(A)$. So assume that each contain both $b^D$ and $c^D$. If $\{b^D,c^D\}$ is a basis of $H$, $h^{-1}b^Dh=u_1(b^D,c^D)$ and $h^{-1}c^Dh=u_2(b^D, c^D)$ for some words $u_1,u_2$, and $h^{-1}w_2(b^D,c^D)h=w_2(u_1,u_2)$. Furthermore, since $u_1$ and $u_2$ are also $(\pi,\epsilon)$-quasigeodesics, $\|u_i\|_H\leq\|u_i\|_{\Gamma}\leq\pi(D_1+2C)+\epsilon< \kappa$ for $i=1,2$. So if $d\geq \kappa$, every power of $b^D$ and $c^D$ in $\mathcal{C}_j$ and $\mathcal{C}_{j'}$ is at least $\kappa> 2M+3$. Furthermore, since $\|w_1\|\geq2\kappa D_1$ and $\|w_2\|\geq2D_1(2M+3)$, $w_1$ contains a power of $b^D$ or $c^D$ which is at least $\kappa$ and $w_2$ contains a power of $b^D$ or $c^D$ which is at least $2M+3$. So $w_1,w_2,u_1,u_2\in H$ satisfy the conditions for Lemma \ref{prod:free} and either $h^{-1}c^Dh$ is equal to the conjugate of some power of $c^D$, or $h^{-1}b^Dh$ is equal to the conjugate of some power of $b^D$, by some word $v(b^D,c^D)$ (since no power of $b$ is conjugate to a power of $c$). In either case $v(b^D,c^D)h^{-1}$ belongs to either $\langle b^D\rangle$ or $\langle c^D\rangle$, since cyclic subgroups of $\Gamma$ are malnormal (\cite{[BH09]}). Therefore, $h\in\langle b^D,c^D\rangle$ and $h^{-1}w_2(b^D,c^D)h=w_1(b^D,c^D)$ in $F(A)$.

 If $\{b^D,c^D, h^{-1}b^Dh\}$ forms a basis of $H$, then $h^{-1}c^Dh$ must be a word in $b^D, c^D,h^{-1}b^Dh$. Assume $\pi$ and $\epsilon$ have been taken so that every path in $Cay(\Gamma)$ labeled by powers of $b,c$ or $h^{-1}b^Dh$ is a $(\pi,\epsilon)$-quasigeodesic for any $h\in B_{\Gamma}(C)$. So again by Lemma \ref{prod:free}, $h^{-1}c^Dh$ is a conjugate of a power of $c^D$, or $h^{-1}b^Dh$ is a conjugate of a power of $b^D$, by some $v(b^{D},c^D, h^{-1}b^Dh)$. Again, either case implies that $vh^{-1}$ is in $\langle b^D\rangle$ or $\langle c^D\rangle$, and so $h\in\langle b^D, h^{-1}b^Dh, c^D\rangle$. But then the equality $h^{-1}w_2(b^D,c^D)h=w_1(b^D,c^D)$ in $H$ implies that $h$ is expressed in terms of  $b^D$ and $c^D$ only, contradicting that $\{b^D,c^D, h^{-1}b^Dh\}$ is a basis of $H$.

Now $h^{-1}w_2h=w_1$ in $F(A)$ implies that there are $\bar w_h,c_1,c_2,c'_1,c'_2\in F(A)$ such that $w_1\equiv c_1\bar w_hc_2$, $w_2\equiv c_1'\bar w_hc_2'$, $\bar w_h$ is the label of a subpath of both $p_x$ and $p'_x$ which begins and ends at $D$-vertices, and $\|c_1\|,\|c_2\|,\|c_1'\|,\|c_2'\|\leq C+D_1$.

Suppose $\bar w_h$ and $\bar w_{h'}$ are labels of shared subpaths of $p_x$ and $p'_x$, constructed in this manner from different repeated labels $h$ and $h'$, of $D$-geodesics. The subpaths labeled by $\bar w_h$ and $\bar w_{h'}$, call them $p_h$ and $p_{h'}$ respectively, do not overlap. Since $D$-geodesics to $D$-vertices on $p_h$ are trivial and $h'$ may emanate from an end vertex of $p_h$ but not from an internal vertex, the subpaths of $p_x$ and $p'_x$ between the furthest apart $D$-geodesics labeled by $h'$, must be disjoint from $p_h$. Since $p_{h'}$ is itself a subpath of those subpaths, it must be disjoint from $p_h$ as well. The bound $\mathcal{E}_1$ given on the sum of length of subpaths of $p_x$ between consecutive $p_h$ and $p_{h'}$ follow from the maximum number of $D$-geodesics without repeated labels at least $2\kappa D_1$ apart, in addition to each of the $\|c_1\|,\|c_2\|$, as above, and the maximum length of the subpaths at the beginning and end of $p_x$ which may not have $D$-geodesics. The bound $\mathcal{E}$ follows from $p'_x$ being a $(\pi,\epsilon)$-quasigeodesic.
\end{proof}

We may now construct the combinatorial generalized equation $\hat E(S_i[\zeta])_{\tilde\Omega}$ from $\bar  E(S_i[\zeta])_{\tilde\Omega}$ using these lemmas. By Corollary \ref{same:partition}, each constant base $\eta$ in $\bar  E(S_i[\zeta])_{\tilde\Omega}$ labeled by $\theta_m([a_{\zeta},b^{dDr_j}]);1\leq j\leq s$ or $\theta_m([a_{\zeta}^N,b^{dDB}])$ may be relabeled by the corresponding $\mathcal{C}_j$. The variable bases partioning $\eta$ may be mapped to the labels of subpaths of $\hat p_j$ between each of $\hat q_{j_0},\ldots,\hat q_{j_{n(j)}}$. While this map will generally not be a solution of this generalized equation, adding some more bases will yield a solution. 

Let $\mathcal{L}=\max\{\mathcal{E},\pi(L+2R)+\epsilon,\pi(L_1+2R)+\epsilon\}$. Given a solution $\psi$ of $\bar  E(S_i[\zeta])_{\tilde\Omega}$, for any variable base $\mu$ (which must correspond to a single variable $x$), if $\|\mu^{\psi}\|=\|x^{\psi}\|>\max\{L_1,L\}$, then cut $\mu$ and $\bar\mu$ each into $2\mathcal{R}(x)+1$ new bases with $\mathcal{R}(x)$ as in Lemma \ref{bnd:error}. The bases corresponding to the $c_{j_k}$ and $c'_{j_{k'}}$, for $0\leq k,k'\leq\mathcal{R}(x)$, are constant bases labeled by those coefficients, which are each of length at most $\mathcal{L}$. The remaining are pairs of dual bases corresponding to $\bar w_{x_j}$ for $1\leq j\leq\mathcal{R}(x)$.

Now suppose $\|\mu^{\psi}\|\leq\max\{L_1,L\}$, with $\mu^{\psi}$ and $\bar\mu^{\psi}$ appearing in $\bar w^{(j)}$ and $\bar w^{(j')}$, where $1\leq j,j'\leq s+1$, as $x_{j_k}^{\psi}$ and $x_{j'_{k'}}^{\psi}$ respectively. Then replace $\mu$ with a constant base labeled by the label of the subpath of $\hat p_j$ from $\hat q_{j_{k-1}}$ to $\hat q_{j_k}$, and replace $\bar\mu$ with a constant base labeled by the label of the subpath of $\hat p_{j'}$ from $\hat q_{j'_{k'-1}}$ to $\hat q_{j'_{k'}}$. Note that each of these subpaths is of length at most $\mathcal{L}$.

Suppose $\|\mu^{\psi}\|\leq\max\{L_1,L\}$, with $\mu^{\psi}$ and $\bar\mu^{\psi}$ appearing in $\bar w^{(j)}$ (as $x_{j_k}^{\psi}$) and $\bar w^{(j')}$ respectively, where $1\leq j\leq s+1$ and $j'>s+1$. Then $\bar\mu$ is covered by a constant base $\eta$ labeled by $\bar c_{j-s-1}$. Cut $\eta$ at the end boundaries of $\bar\mu$ and for the new constant base $\eta_1$ covered by $\bar\mu$, let $\sigma(\eta_1)$ be equal to the label of the subpath of $\hat p_j$ from $\hat q_{j_{k-1}}$ to $\hat q_{j_k}$ (which is of length at most $\mathcal{L}$).

Finally, if a pair of dual bases are covered by constant bases labeled by $\bar w^{(j)}$ and $\bar w^{(j')}$ for $j,j'>s+1$, no changes are needed. By Lemma \ref{bnd:error}, this new generalized equation $\hat E(S_i[\zeta])$ has a solution $\hat\psi$. Furthermore, disks labeled by the coefficients of $\hat E(S_i[\zeta])$ may be glued together on the same union of spheres $\Sigma$ as $\bar D$. This follows from the construction of $\hat E(S_i[\zeta])$, which allows the disks of $\bar D$ to be relabeled with coefficients of $\hat E(S_i[\zeta])$ so that some gluing between disks may be removed but new disks are introduced which are glued in as ``$0$-cells'' (in the terminology of \cite{Olshanskii-1989}), maintaining a tiling on $\Sigma$. Denote the disk diagram obtained by this process as $\hat D$. 

Recall that $\bar E(S_i[\zeta]_{\tilde\Omega})$ has at most $15s_1-s-1$ variable bases and $3s_1+s+1$ constant bases. So there are $K_2\leq(B_C+1)(15s_1-s-1)+3s_1+s+1$ many constant bases in addition to those labeled by $\mathcal{C}_1,\ldots,\mathcal{C}_{s+1}$, and there are $K_1\leq B_C(15s_1-s-1)+2K_2$ many variable bases in $\hat E(S_i{\zeta}_{\tilde\Omega})$. Rename the $K_2$ coefficients of constant bases (and labels of disk boundaries) other than those labeled by $\mathcal{C}_1,\ldots,\mathcal{C}_{s+1}$, by $R_1,\ldots,R_{K_2}$. Proposition \ref{coms:diagram} is proved.

\end{proof}

We now show that for certain $\zeta$, the disk diagram $\hat D$ constructed in Proposition \ref{coms:diagram} is equivalent (in that it has the same labels of disk boundaries and tiles the same union of spheres) to another disk diagram with particular properties.

\begin{lemma}\label{a-bands} There is a number $n=O(s^2)$, and a tuple $(t_1,\ldots,t_n)$ of positive integers, such that for $\kappa>K_0$ as in Lemma \ref{bnd:error} and $d>\max\{\kappa,\mathcal{L}\}$, then if $S[\zeta]$ has a solution in $\Gamma$ for the corresponding $\zeta$, there is a disk diagram equivalent to $\hat D$, in which
words $a_{\zeta}$ are only glued to words $a_{\zeta}$.
\end{lemma}  

\begin{proof}By Proposition \ref{coms:diagram}, there is a solution $\psi$ of $\hat E(S[\zeta])$. There are $O(s)$ boundaries from variable and constant bases in $\hat E(S[\zeta])_{\tilde\Omega}$ for $\zeta$ constructed from any $n$. So there is some $n=\mathcal(s^2)$ such that $a_{\zeta}=b^{\kappa D}c^{dt_{1}D}b^{\kappa D}\ldots b^{\kappa D}c^{dt_{n}D}b^{\kappa D}$ has a subword $w_1=b^{\kappa D}c^{dt_iD}b^{\kappa D}$ which is not cut by the solution of a variable base, in all occurrences of $a_{\zeta}$ in $\mathcal{C}_1,\ldots,\mathcal{C}_{s+1}$.
Now let $t_1>K_2\mathcal{L}$, $t_{j+1}>t_j+K_2\mathcal{L}$ for $1\leq j \leq n-1$. By Proposition \ref{coms:diagram}, having no cuts from solutions of variable bases on the subword forces an exact gluing of that subword on disk boundaries of $\hat D$, up to a possible error disk with boundary labeled $R_j$. But since $d>\mathcal{L}>|R_j|$, the $R_{1},\ldots, R_{K_2}$ are all too short to be glued to the subword as a $0$-cell. If $R_j$ is glued to both occurences of the subword, then $R_j$ would be an unreduced word. This implies that every section of a disk boundary labeled by $a_{\zeta}$ has the portion labeled by $w_1$ glued to a portion labeled by $w_1$ (since it only appears once in $a_{\zeta}$) of another section of a disk boundary labeled by $a_{\zeta}$.

It is possible to find an equivalent diagram such that all occurrences of words $a_{\zeta}$ are glued only to words $a_{\zeta}$.  In particular, for a variable base $\mu$ covering $w_1$, assume that $\mu^{\psi}=w_1$ (the covering is exact) by cutting the bases $\mu,\bar\mu$ if necessary (as in the rewriting process, giving a finite collection of generalized equations, one of which has a solution). So $\bar\mu$ also exactly covers $w_1$.

If $w_1$ is not the terminal subword of $a_{\zeta}$ (i.e. $i\neq n$), let $\mu_1$ be the variable base with $\alpha(\mu_1)=\beta(\mu)$. If $\epsilon(\bar\mu)=\epsilon(\mu)$, let $\mu_2$ be the variable base with $\alpha(\mu_2)=\beta(\mu)$, whereas if $\epsilon(\bar\mu)=-\epsilon(\mu)$, then let $\mu_2$ be the variable base with $\beta(\mu_2)=\alpha(\mu)$. If $\|\mu_1^{\psi}\|=\|\mu_2^{\psi}\|$ then $\mu_1$ and $\mu_2$ cover the same subword $w'_1$ of $a_{\zeta}$, where $w'_1$ directly precedes or follows $w_1$. So there is an exact gluing of a strictly larger subword $w_2$ equal to $w_1w'_1$ or $w'_1w_1$.
 
Without loss of generality, if $\|\mu_1^{\psi}\|>\|\mu_2^{\psi}\|$, then it is possible to cut $\mu_1$ into bases $\nu_1\nu_2$ (and $\bar\mu_1$ into $\bar\nu_1\bar\nu_2$) so that $\|\nu_1^{\psi}\|=\|\mu_2^{\psi}\|$ and again there must be an exact gluing of the strictly larger subword $w_2$ of $a_{\zeta}$ covered by $\mu\nu_1$.

Similarly, if $w_1$ is not the inital subword of $a_{\zeta}$ (i.e. $i\neq1$), let $\mu'_1$ be the variable base with $\beta(\mu'_1)=\alpha(\mu)$ and $\mu'_2$ be the variable base analagously defined in reverse to above. The same process again forces an exact gluing of a strictly larger subword $w_2$ of $a_{\zeta}$ covered by $\mu'_1\mu$ or $\nu'_1\mu$.

Iterating these processes for the resulting subwords $w_2$ (using $\beta(\mu_1)$ or $\beta(\nu_1)$ instead of $\beta(\mu)$ and $\alpha(\mu'_1)$ or $\alpha(\nu'_1)$ instead of $\alpha(\mu)$ accordingly) for every base $\mu$ covering the subword $w_1$, we obtain a diagram equivalent to $\hat D$ with every label $a_{\zeta}$ exactly glued to another label $a_{\zeta}$.\end{proof}
 \begin{prop}\label{removeError}Suppose $d>\max\{\kappa,K_2\mathcal{L}\}$, where $K_2$ is as in Proposition \ref{coms:diagram}, $\kappa>K_0$ satisfies Lemma \ref{bnd:error}, and $(t_1,\ldots,t_n)$ satisfying Lemma \ref{a-bands}. Then if $S[\zeta]$ has a solution in $\Gamma$ for the corresponding $\zeta$, the equation
  \begin{equation}\label{***}\prod_{j=1}^sz_j^{-1}[a_1,b_1^{r_j}]z_j=[a_1^{N},b_1^{B}]\end{equation} has solution in 
  $F(a_1,b_1)$ for $a_1=a_{\zeta}$ and $b_1=b^{dD}$.\end{prop}
  \begin{proof}Suppose $d>\max\{\kappa,K_2\mathcal{L}\}\geq K_2\mathcal{L}>\mathcal{L}$. Then words $a_{\zeta}$ are only glued to words $a_{\zeta}$ (``$a_{\zeta}$-bands'' are formed, as in \cite{KLMT}) in the diagram constructed in Lemma \ref{a-bands}. Since $\sum_{j=1}^{K_2}\|R_j\|\leq K_2\mathcal{L}<d<\|a_{\zeta}\|$, there are no annuli of $a_{\zeta}$-bands with some disks filling in the center that have labels from $R_1,\ldots,R_{K_2}$. So disks with boundaries $R_j$ that glue to $a_{\zeta}$-bands must be glued between the sides of $a_{\zeta}$-bands labeled by powers of $b$. However, since $b$ is not a proper power it can not be shifted relative to itself (i.e. no path labeled by $b$ may end in the middle of another path labeled by $b$), so any $R_j$ disk between two $a_{\zeta}$-bands partially glued together, must be labeled by powers of $b$ and exactly the same subwords of $b$ (starting from where the two $a_{\zeta}$-bands separate). In other words $R_j=b_1b_1^{-1}b^mb_2b_2^{-1}b^{-n}$ for subwords $b_1,b_2$ of $b$ or $b^{-1}$, and integers $m,n$. But the labels $R_j$ must be reduced words, so these disks must only glue to other disks labeled $R_{j'}$. Therefore the $R_j$-disks appear in the diagram tiling spheres that have no $a_{\zeta}$-bands. Those spheres may be removed, giving a disk diagram with spheres tiled by only disks with labels $\mathcal{C}_1,\ldots,\mathcal{C}_{s+1}$, glued as $a_{\zeta}$-bands. This new diagram may be labeled with powers of $a_{\zeta}$ and $b^{dD}$, maintaining gluings which respect those labels. So by \cite{Olshanskii-1989}, equation (\ref{***}) has a solution in $F(a_{\zeta},b^{dD})$. 
  \end{proof}

As in \cite{KLMT}, this tiling of a sphere gives a solution to the given bin packing input by filling a disk labellabeled by $[a_1^N,b_1^B]$ with $a_1$-bands.
\begin{cor}For such $\zeta$, if $S[\zeta]=1$ has a solution in $\Gamma$, then there a solution to the given bin packing input.
\end{cor}
\begin{cor}The Diophantine problem for quadratic equations over a non-cyclic torsion-free hyperbolic group, is NP-hard
\end{cor}
 On the other hand, Theorem \ref{th:1}
implies that this problem is in NP. The proof of Theorem \ref{NPhard} is now complete.
\\
\\
{\bf Acknowledgments}

In conclusion, we thank the referee whose remarks have greatly improved the exposition. The first author acknowledges the support by NSF grant DMS-0700811, the fourth author acknowledges the support by  EPSRC Grant EP/K016687/1.  The authors thank Erwin Schr{\"o}dinger International Institute for Mathematical Physics and the program 
``Geometry of computation in groups'', where they were able to discuss the results of this paper in April 2014.

\bibliography{quadr_eq_hyp_biblio-2}
\bibliographystyle{amsplain}

Olga Kharlampovich,  (okharlampovich@gmail.com) Dept. Math and Stats, Hunter College, CUNY, 695 Park Avenue
New York, NY 10065 USA.

Atefeh Mohajeri, (at.mohajeri@gmail.com) Dept. Math and Stats, McGill University, 805 Sherbrooke St. W., Montreal, Canada, H3A 0B9.

Alexander Taam,  (alex.taam@gmail.com) Dept. Math, Graduate Center, CUNY, 365 Fifth Avenue
New York, NY 10016 USA.

Alina Vdovina, (Alina.Vdovina@newcastle.ac.uk) Dept. Math and Stats, Newcastle University, Newcastle University 
Newcastle upon Tyne, NE1 7RU, UK.
\end{document}